\documentclass[11pt,a4paper]{article}
\usepackage[utf8]{inputenc}
\usepackage[english]{babel}
\usepackage{amsmath}
\usepackage{amsthm}
\usepackage{amsfonts}
\usepackage{amssymb}
\usepackage{graphicx}
\usepackage{booktabs}
\usepackage[hidelinks]{hyperref}
\usepackage{subcaption}
\usepackage{multirow}
\usepackage{tikz}
\usepackage{verbatim}
\usepackage[ruled,vlined]{algorithm2e}
\usepackage{enumitem}
\usepackage{float}
\usepackage{algorithmic}

\usetikzlibrary{arrows,shapes,positioning}
\tikzstyle{block} = [draw, rectangle, minimum width = 0.75cm, minimum height = 0.75cm]
\tikzstyle{sum} = [draw, circle, minimum size=.5cm, node distance=1.75cm]
\tikzstyle{input} = [coordinate]
\tikzstyle{output} = [coordinate]

\numberwithin{equation}{section}

\newtheorem{theorem}{Theorem}[section]

\newtheorem{lemma}{Lemma}[section]

\makeatletter
\newcommand{\startat}[1]{\setcounter{\@enumctr}{#1}%
	\addtocounter{\@enumctr}{-1}}
\makeatother

\def\R {{\mathbb R}}

\def\N {{\mathbb N}}

\def\cH {{\cal H}}

\def\cH {{\cal H}}

\def\hJ {\hat{J}}

\newcommand{\normiii}[1]{{\vert\kern-0.25ex\vert\kern-0.25ex\vert #1 
		\vert\kern-0.25ex\vert\kern-0.25ex\vert}}

\makeatletter
\newcommand{\mylabel}[2]{#2\def\@currentlabel{#2}\label{#1}}

\DeclareMathOperator*{\dive}{div}

\title{\bf Optimal design of equilibrium solutions of the Vlasov-Poisson system by an external electric field}
\author{A. Borz{\`i}
	\thanks{Institut f\"ur Mathematik, Universit\"at W\"urzburg, Emil-Fischer-Strasse 30,
		97074 W\"urzburg, Germany. e-mail: alfio.borzi@mathematik.uni-wuerzburg.de}
	\and 
	G. Infante
	\thanks{Dipartimento di Matematica e Informatica, Universit\`a della Calabria, 87036 Arcavacata di Rende, Cosenza, Italy
		e-mail: gennaro.infante@unical.it}
	\and
	G. Mascali
	\thanks{Dipartimento di Matematica e Informatica, Universit\`a della Calabria, 87036 Arcavacata di Rende, Cosenza, Italy, and INFN gruppo collegato Cosenza, Italy
		e-mail: giovanni.mascali@unical.it}
}
\date{\today}

\begin{document}

 \maketitle

 \begin{abstract}

A new optimization framework to design steady equilibrium solutions of the 
Vlasov-Poisson system by means of external electric fields is presented. 
This optimization framework requires the minimization of an ensemble functional 
with Tikhonov regularization of the control field 
under the differential constraint of a nonlinear 
elliptic equation that models equilibrium solutions of the Vlasov-Poisson system. 
Existence of optimal control fields and their characterization as solutions 
to first-order optimality conditions are discussed. Numerical approximations 
and optimization schemes are developed to validate the proposed framework.

 \end{abstract}
 


 \bigskip
 	\section{Introduction}
 	\label{sec:intro}
 	
 	The need to confine and manipulate plasma appears in a multitude of applications 
 	ranging from plasma cooling \cite{Manfredi2012} to satellite electric propulsion 
 	\cite{Mazouffre2016} and energy production in fusion reactors \cite{Dunlap2021}. 
 	In particular, in many of these applications, there is the need to determine 
 	steady equilibrium configurations having specific properties useful for the 
 	given application. For this reason, many different strategies have been 
 	proposed and analyzed \cite{Doyle2021, Heumann2018,Holst2022}. 
 	
 	We would like to contribute to this field of research with the formulation 
 	of a novel optimization framework for optimal design of equilibrium 
 	configurations that we exemplify considering a 
 	Vlasov-Poisson system subject to an external electric control field. The main novelties  
 	of our approach are: 1) The introduction of an ensemble objective functional 
with a first-order Tikhonov regularization of the external control field sought; 
2) The choice of the differential constraint given by a nonlinear elliptic equation that models the 
equilibrium solutions of the Vlasov-Poisson system with a given external field; 
3) The complete theoretical analysis of the resulting optimization problem; 
4) The numerical validation of the proposed optimal design procedure.

The Vlasov-Poisson system governs the transport of a particle distribution function 
$f$ under the action of a self-consistent force whose potential is 
modelled by a Poisson equation corresponding to electrostatic forces between charged particles with the same sign. Existence and uniqueness 
of classical solutions has been proved for dimensions $n \le 3$ under mild assumptions on the initial data in \cite{DiPerna1988,Perthame1991,Pfaffelmoser1992}, 
and weak solutions are discussed in \cite{Arsenev1975,Horst1984}. 
In many situations, the Vlasov-Poisson system represents an accurate alternative 
to the Boltzmann description of gases of charged particles whose 
Coulomb interaction cannot be considered to be of short-range type. 
We choose the Vlasov-Poisson system since it is a valid model in describing various phenomena in plasma physics and therefore it is representative for the class 
of problems considered in our work. 

Specifically, we consider steady equilibria solutions of 
the Vlasov-Poisson system that are characterized by the solution to the 
nonlinear elliptic Poisson-Boltzmann equation, see \cite{Bavaud1991,Carrillo1998,Desvillettes1991,GoLi1989}, 
including the presence of an external electric control field. Our purpose is 
to design this control field such that the corresponding steady equilibrium distribution 
concentrates on some selected regions of the computational domain. 

{Our work is timely, considering very recent efforts in the investigation of optimal 
	control problems governed by the Vlasov-Poisson system. In particular in \cite{Albi2024,Einkemmer2024} the focus is in constructing feedback control mechanisms with electric and magnetic fields that ensure dynamical 
	stability around closed trajectories of desired plasma dynamics; see also \cite{BartschB2024} for a related problem. Further, we refer 
	to \cite{BartschC2024,Knopf2018,Weber2021} for recent numerical and theoretical results concerning open-loop time-dependent optimal control problems 
	governed by Vlasov-Poisson models. We remark that all these works propose 
	methodologies that allow to drive the system to approximately reach a desired target 
	configuration. In this context, our methodology provides the tools for designing 
	target states that are equilibria under the action of a fixed external field.  
}

For our purpose, we formulate a constrained optimization problem, whose differential 
constraint is the nonlinear Poisson-Boltzmann equation with the external control 
field to be optimized by minimizing a regularized ensemble objective functional. 
Our choice of an ensemble functional is dictated by the statistical nature 
of the underlying Vlasov-Poisson system and differs considerably from 
$L^2$ error tracking functionals as the one considered in, e.g., \cite{Heumann2018,Holst2022}, 
where a desired distribution profile must be provided that could be 
unattainable for the given physical setting. In fact, our ensemble functional 
corresponds to a statistical expected value functional of a valley-potential 
in which we wish to concentrate the particle density. This class 
of functionals was introduced in \cite{Brockett2012} and recently 
applied and analyzed in \cite{BartschNastasi2021, Bartsch2019,Bartsch2021}. 
Moreover, in order to guarantee well posedness of our optimization 
problem and the appropriate regularity of the resulting optimal 
control field, our objective functional includes a first-order Tikhonov 
regularization term; see, e.g., \cite{Aster2013} for a discussion 
on Tikhonov regularization of different orders. 
{We remark that our framework is able to accommodate other objective 
	functionals that involve different 
	moments (e.g., mean, variance, etc.) of the distribution function solving the Poisson-Boltzmann equation. }
	
In this work, we theoretically investigate the proposed optimization 
problems proving existence of optimal control fields. Further, we 
discuss the characterization of these solutions by first-order optimality 
conditions that result in an optimality system consisting of 
the Poisson-Boltzmann model, its optimization adjoint, and 
an optimality gradient condition. In this way, we are able to identify 
the gradient of the objective functional subject to the differential 
constraint, which is required in order to construct a numerical 
optimization procedure aiming at solving the optimality system. 
For this purpose, since the control field is sought in the Sobolev 
space of $H^1$ functions, we discuss the construction of the reduced 
gradient in this Hilbert space.

 	In next section, we illustrate the Vlasov-Poisson system that models the 
 	evolution of the charged particles distribution function in phase space. Further, 
 	we derive the nonlinear Poisson-Boltzmann elliptic equation for steady equilibrium distributions corresponding to a given external electric field. In Section \ref{sec:control}, we identify this electric 
 	field as our control function and formulate our optimal external electric field 
 	problem. For this purpose, we introduce an ensemble 
 	objective functional, including a $H^1$ Tikhonov regularization term, which 
 	has to be minimized subject to the differential constraint given by the 
 	nonlinear elliptic equation. In this section, the regularity of the 
 	control-to-state map is discussed in detail, and existence 
 	of optimal solutions is proved. In Section \ref{sec-optsys}, we discuss 
 	the characterization of optimal control fields by first-order optimality 
 	conditions. In particular, we identify the reduced gradient that is used 
 	in our iterative solution procedure. 
 	
 	In Section  \ref{sec-gradient}, we discuss the numerical implementation of our optimization framework that requires 
the numerical approximation of the Poisson-Boltzmann equation 
and of its optimization adjoint, and the assembling of the $H^1$ gradient 
that is used in a gradient-based iterative solution procedure. 
Further, we illustrate a nonlinear conjugate gradient algorithm with $H^1$ gradient
 	to compute the optimal control fields. In Section \ref{sec-numexp}, we 
 	report results of numerical experiments that demonstrate the ability 
 	of our optimization framework to manipulate the steady equilibrium 
 	plasma configuration by an external electric field. A conclusion Section 
 	completes this work.

 	 	\section{The Vlasov-Poisson system and \\its equilibrium solutions}
 	\label{sec:model}
 	
 	The Vlasov equation \cite{Vlasov1967} is a partial differential equation 
 	that models the evolution in the phase space of the distribution function of a plasma of charged interacting particles. In the zero-magnetic field limit, 
 	the Vlasov equation for charged particles subject to a given external electric field reads
 	\begin{equation}
 	\partial_t f  + v \cdot \nabla_x f + \big( E(x,t) + E_0(x) \big) \cdot \nabla_v f =0 ,
 	\label{eVlasov01}
 	\end{equation}
 	where $f=f(x,v,t)$ denotes the distribution function, with $x$, $v$ and $t$ respectively being position, velocity and time, $E=E(x,t)$ represents the self-consistent (internal) electric field generated by the particles, and 
 	$E_0=E_0(x)$ is a stationary external electric field. 
 	
 	We are dealing with the case of positively charged particles, for which the self-consistent field is defined as the solution to the 
 	following Poisson problem
 	\begin{equation}
 	\dive E (x,t) = \rho(x,t) , \qquad   \rho(x,t) = \int_{\R^n} f(x,v,t) \, dv , 
 	\label{eExt01}
 	\end{equation}
 	for any fixed time $t \ge0$; we denote with $\dive$ the divergence in $x$, 
 	while $\rho$ is the spatial particle density, and $n$ is the spatial dimension. 
 	
 	For both electric fields $E$ and $E_0$, we assume that there exist 
 	differentiable potentials $U$ and $u$, respectively, such that it holds 
 	\begin{equation}
 	 	     E (x,t) = - \nabla_x U(x,t) , \qquad E_0 (x) = - \nabla_x u (x) .
\label{eEnablaU01}
 	\end{equation}
 We also assume that the particles of the plasma are spatially localized within a 
 	bounded and convex set $\Omega \subset \R^n$ with smooth boundary, 
 	and the mean velocity of this system of particles is zero. 
 
It is well known \cite{Bavaud1991,Desvillettes1991,Dolbeault1991} that 
these conditions are sufficient for proving existence of stationary 
Maxwellian solutions to \eqref{eVlasov01} with the following structure 
\begin{equation}
f_s(x,v) = \frac{1}{(2 \pi T)^{n/2}} \, \rho(x) \, e^{- \frac{|v|^2}{2T}} ,
\label{eDensStat01}
\end{equation}
where $T>0$ denotes the constant temperature of the plasma; 
{we assume particles of unit mass and the Boltzmann constant 
is set equal to 1.}

Notice that \eqref{eDensStat01} is completely specified by $\rho$ 
that satisfies the following system
$$
\nabla_x \rho = \frac{1}{T} \, (E  + E_0 ) \, \rho, \qquad \dive E = \rho \, ;
$$
see \eqref{eVlasov01} and \eqref{eExt01}. Thus, using \eqref{eEnablaU01}, 
it appears sufficient to find $\rho$ that satisfies the following equations
\begin{equation}
\nabla_x \log \rho = - \frac{1}{T} \, ( \nabla_x U + \nabla_x u ) ,
 \qquad -\Delta U = \rho ,
 \label{eRhoE01}
\end{equation}
where $\Delta$ represents the Laplacian operator in $x$. With the first 
equation in \eqref{eRhoE01}, we arrive at a formula for $\rho$ as a function of the potentials, 
which including $L^1$ normalization, is given by 
\begin{equation}
\rho(x) = \frac{e^{-(U(x)+u(x))/T}}{\int_\Omega e^{-(U(x)+u(x))/T }\, dx}.
\label{eRHO01}
\end{equation}
Clearly, in the stationary case the dependence on time of $E$, respectively $U$, 
is also dropped. We remark that, with the chosen normalization, we 
can interpret $\rho$ as a probability density function. 

With the second 
equation in  \eqref{eRhoE01}, and \eqref{eRHO01}, we obtain the 
nonlinear Poisson-Boltzmann equation for $U$, given $u$, as follows: 
\begin{equation}
-\Delta U (x) =  \frac{e^{-(U(x)+u(x))/T}}{\int_\Omega e^{-(U(x)+u(x))/T }\, dx}.
\label{eEqU01}
\end{equation}

It could be convenient to introduce {$\rho_0 (x):= e^{-u(x)/T}$}, and with 
this setting, we get
\begin{equation}
-\Delta U (x) = \frac{\rho_0 (x) \, e^{-U(x)/T}}{\int_\Omega \rho_0 (x) \, e^{-U(x)/T }\, dx} .
\label{eEqU02}
\end{equation}

Further, we can take $\Omega$ sufficiently large so that 
$U=0$ can be assumed on $\partial \Omega$. For simplicity, 
in the following we also set $T=1$. 
 
The existence and uniqueness of solutions to \eqref{eEqU01} with 
homogeneous Dirichlet boundary conditions has been studied with variational methods in \cite{Desvillettes1991,GoLi1989,Dolbeault1991}. In the following, for the sake of completeness, 
we give a sketch of the proof by means of an alternative approach.
{
\begin{theorem}
	Let $\rho_0\in L^1(\Omega)$, then \eqref{eEqU01} with homogeneous Dirichlet boundary conditions has a solution $U\in H^2(\Omega)$.
\end{theorem}
\begin{proof}
Let us start introducing {$U_0\equiv 0$}, and $U_k$, $k=1, 2, \dots$, in such a way that
$$
-\Delta U_k(x)= \frac{e^{-(U_{k-1}(x)+u(x))/T}}{\int_\Omega e^{-(U_{k-1}(x)+u(x))/T }\, dx}=:f_k(x).
$$
The existence of these functions is assured by results on linear second order elliptic equations \cite{Evans2010}; moreover, we have $||f_k||_{L^1(\Omega)}=1$, $k\in \mathbb{N}$. Note that,
since $L^2(\Omega)$ is continuously embedded into  $L^1(\Omega)$, it follows that the sequence $\{f_k\}_{k=1}^\infty$ is bounded in $L^2(\Omega)$. Regularity results on linear second-order elliptic equations \cite{Evans2010} imply that there exists a constant $C>0$ such that $||U_k||_{H^2(\Omega)}\le C$, $k\in \mathbb{N}$. Therefore there exists a function $U\in H^2(\Omega)$ such that $U_k\rightharpoonup U$ in $H^2(\Omega)$, after extraction of a subsequence, if necessary. Passing to the limit for $k\rightarrow +\infty$, it is possible to prove that $U$
solves the equation under consideration.
\end{proof}}   
 
Concerning the uniqueness of solutions to our elliptic problem, we exploit the same technique used in \cite{GoLi1989} for proving the uniqueness in the case without an external field.  {For the sake of clarity, we premise the following 
\begin{lemma}
	Let $U_1$ and $U_2$ be two solutions of \eqref{eEqU01}, with homogeneous Dirichlet boundary conditions, and set
	$$
I_i:=\int_\Omega \exp{(-u-U_i)}\,dx, \,\, i=1,2. 
$$
If $I_1\ge I_2$ then $U_1(x)\le U_2(x), \forall x\in \Omega.$ In particular if $I_1=I_2$, then $U_1(x)= U_2(x), \forall x\in \Omega.$
\end{lemma}	
\begin{proof}
Notice that $-\Delta U_i>0,\,\,i=1,2$, since $\frac{\exp{(-u-U_i)}}{I_i}>0,\,\, i=1,2$, therefore $U_1$ and $U_2$ are concave. Suppose that $I_1>I_2$
and {$\max_{x\in \Omega} \Big(U_1-U_2\Big)>0$}, and let $\overline x\in \Omega$ be a point of maximum for $U_1-U_2$, then 
by the necessary condition for local maximum $-\Delta \Big(U_1-U_2\Big)(\overline x)\ge 0.$ Since $I_1>I_2,$ we have
$$
\Bigg(-\frac{I_1}{I_2}\Delta U_1+\Delta U_2\Bigg)(\overline x)\ge \Bigg(-\Delta U_1+\Delta U_2\Bigg)(\overline x)\ge 0.
$$
On the other hand
$$
\Bigg(-I_1 \Delta U_1+ I_2 \Delta U_2\Bigg)(\overline x)=
\exp{(-u(\overline x))}\Bigg[\exp{(-U_1(\overline x))}-\exp{(-U_2(\overline x))}\Bigg] <0,
$$
which is in contradiction with the previous inequality. \\
We have shown that if $I_1>I_2$, then $U_1(x)\le U_2(x),\,\,\forall x\in \Omega.$ By the same reasoning, one can prove that if $I_1=I_2$,  then 
$U_1(x)=U_2(x), \,\,\forall x \in \Omega$.
\end{proof}
Now, we are in a condition to prove the uniqueness. 
\begin{theorem}
	Equation \eqref{eEqU01}, with homogeneous Dirichlet boundary conditions, has a unique solution.
\end{theorem}	
\begin{proof}
By contradiction, let $U_1(x)$ and  $U_2(x)$ be two solutions of \eqref{eEqU01}, and
without loss of generality, let us suppose that $I_1>I_2$. Then, as shown in the previous Lemma, we have $U_1(x)\le U_2(x)$.
Suppose that $\min_{x\in \Omega} \Big(U_1-U_2\Big)<0$ and let $\overline x\in \Omega$ be a point of minimum, we 
have, by the necessary condition of local minimum 
\begin{equation}
-\Delta\Big(U_1-U_2\Big)(\overline x)\le 0.
\label{ine0}
\end{equation}
We have also
\begin{equation}
-\Delta\Big(U_1-U_2\Big)(\overline x)=
\exp{(-u(\overline x))}\Bigg[\frac{\exp{(-U_1(\overline x))}}{I_1}-\frac{\exp{(-U_2(\overline x))}}{I_2}\Bigg].
\label{ine1}
\end{equation}
Since
$$
\exp{(-u(x) -U_1(x)+U_1(\overline x))}\le \exp{(-u(x) -U_2(x)+U_2(\overline x))},
$$
the second factor in the right-hand side of \eqref{ine1} must satisfy the following inequality
\begin{eqnarray*}
&& \frac{\exp{(-U_1(\overline x))}}{I_1}-\frac{\exp{(-U_2(\overline x))}}{I_2}
=\frac{1}{\int_\Omega\exp{(-u(x) -U_1(x)+U_1(\overline x))}\,dx}\\
&-&\frac{1}{\int_\Omega\exp{(-u(x)-U_2(x)+U_2(\overline x))}\,dx}\ge 0.
\end{eqnarray*}
There are two possibilities.
\begin{itemize}
	\item The equality holds, then $(U_2-U_1)(x)=(U_2-U_1)(\overline x)$, $\forall x\in \overline \Omega$, and we have a contradiction since
	$(U_2-U_1)(x)=0$ on $x\in \partial \Omega.$
	\item The strict inequality holds, and we again have a contradiction with \eqref{ine0}.
\end{itemize}
This concludes the proof.
\end{proof}}

\section{Optimal control of equilibrium distributions}
\label{sec:control}
Our governing model is the nonlinear elliptic equation \eqref{eEqU01} 
with homogeneous Dirichlet boundary conditions. The solution 
of this boundary-value problem for a given $u$ belonging to a suitable space $D(\Omega)$ to be specified below, defines the map
$$
u \mapsto U.
$$
In turn, the function above and \eqref{eRHO01} define the composed map
$$
u \mapsto \rho = \Phi(U,u), 
$$
where 
\begin{equation}
\Phi(U,u):= \frac{e^{-(U+u)}}{\int_\Omega e^{-(U(x)+u(x))}\, dx} .
\label{eRHO02}
\end{equation}

Our purpose is to formulate optimal control problems that allow us
to construct external potential fields $u$, acting on the plasma, such 
that a distribution function with some desired properties is obtained. For this reason we consider
an objective functional belonging to the class of ensemble 
cost functionals; see, e.g., \cite{BartschNastasi2021, Bartsch2019,Bartsch2021} 
and references therein. Specifically, we focus on the following 
objective functional 
\begin{equation}
	J(\rho, u )= \int_\Omega V(x) \, \rho(x) \, dx +   \frac{\alpha}{2} \, \| u \|^2_{H^1(\Omega)} .
	\label{eJ02}
\end{equation}
In \eqref{eJ02}, we have a weight $\alpha >0 $ that multiplies the $H^1$ cost of the potential $u$ as in \cite{Annunziato2021}. This choice is motivated by the fact that 
a potential in $H^1(\Omega)$ is sought. 
The first term in \eqref{eJ02} models the requirement that, 
by minimization, the density $\rho$ concentrates along the bottom of the `valley' 
function $V$, which we assume to be measurable and bounded. In particular, if  we choose 
\begin{align*}
	V(x) = - A \, \exp\left(-\frac{|x-x_0|^2}{2\, a^2} \right), 
	\qquad A >0, \quad a >0.
\end{align*}
the bottom is at $x=x_0$.
With the same reasoning, it is also possible to aim at designing a multimodal 
$\rho$, or one distributed on a ring, etc..

Another possible objective functional could be given by  
\begin{equation}
	J(\rho, u )= \frac{1}{2} G(\rho,\rho_d) +   \frac{\alpha}{2} \, \| u \|^2_{H^1(\Omega)} ,
	\label{eJ01}
\end{equation}
where $G(\rho,\rho_d)$ denotes an appropriate {measure of the difference} between 
$\rho$ and a desired $L^1$-normalized target distribution $\rho_d$. 
In particular, one could focus on the Kullback-Leibler (KL) 
divergence given by \cite{KullbackLeibler1951}
\begin{equation}
	G (\rho, \, \rho_d ) :=
	\int_\Omega   \rho (x) \, \log(\rho(x) / \rho_d (x))     \, dx .
	\label{gKL}
\end{equation}
We remark that our framework can be extended with only 
minor modification to 
accommodate the KL functional given above.

Now, we focus on the case with \eqref{eJ02} and formulate our class of optimal control problems:
\begin{align}
 \min J(\rho,u) & \nonumber \\
 \mbox{ s.t.}  \quad \rho & =\Phi(U,u),   \nonumber\\
                 - \Delta U &= \Phi(U,u) \qquad \mbox{ in } \Omega, \label{contrpr} \\
                 U &=0  \qquad \qquad \mbox{ on } \partial \Omega.  \nonumber
\end{align}

As said, the solution 
to the boundary-value problem in \eqref{contrpr} for a given $u\in D(\Omega)$ 
defines the so-called control-to-state map \cite{Lions1971,Troeltzsch2010}
\begin{equation}
S: D(\Omega) \rightarrow L^2(\Omega), \qquad  u \mapsto U=:S(u).
\label{eC2Smap}
\end{equation}

Regarding the space of functions $D(\Omega)$, we distinguish the following three cases: 
 {\begin{enumerate}
	\item Space dimension $n=1$, we take
	$$
	D(\Omega)=D_1(\Omega):=	
	H^1_0(\Omega)=\{v\in H^1(\Omega);\,v|_{\partial \Omega}=0\},
	$$		
	where the equality to zero at the boundary has to be intended in a generalized sense.
\end{enumerate}  
If $n=2$ or $n=3$, we investigate two possible strategies:
\begin{enumerate} \startat{2}
	\item  $n=2$ or $n=3$ first strategy, we consider uniformly bounded external potentials, that is, we choose
	\begin{eqnarray*}
		D(\Omega)=D_2(\Omega):=	\{v\in H^1_0(\Omega); \,\, M_1\le v\le M_2,\,\,{\textrm{a.e.}} \,\,{\rm in}\,\,\Omega\}, 
	\end{eqnarray*}
	with fixed $M_1, M_2\in\mathbb{R};$
\item  $n=2$ or $n=3$ second strategy, we require more regularity on the control potential, that is we assume
\begin{eqnarray*}
	D(\Omega)=D_3(\Omega):=	
	H^2_0(\Omega)=\{v\in H^1_0(\Omega) \cap H^2(\Omega); \nabla_x v|_{\partial \Omega}=0\},
\end{eqnarray*}
and correspondingly we change $||u||_{H^1(\Omega)}$ into $||u||_{H^2(\Omega)}$ in \eqref{eJ02}.	
\end{enumerate} }

In order to analyse the control-to-state map \eqref{eC2Smap}, we reformulate our governing model in terms 
of the Green's function of the Laplace operator in the given domain and with
Dirichlet boundary conditions. We have
$$
U(x)= \int_\Omega k(x,y) \,  \Phi(U,u)(y) \, dy ,
$$
where $k(\cdot , \cdot )$ denotes the Green's function. 

{
	\begin{theorem} 
		For any Lipschitz domain $\Omega$ with Poincar\'e constant 
		$c_\Omega<4$ the control-to-state map is Fr\'echet differentiable. 
		\label{teo1}
	\end{theorem} 
	\begin{proof}
		The differentiability of $S$ is proved by applying the implicit function theorem 
		to the operator 
		$$
		c(U,u):= U(x) - \int_\Omega k(x,y) \,  \Phi(U^+,u)(y) \, dy,
		$$
		where $U^+(x):=\max(0,U(x))$. 
		Notice that $c(U,u)=0$ defines in an implicit way the control-to-state map, since the solutions to the boundary value problem in \eqref{contrpr} are almost everywhere nonnegative; see \cite{Desvillettes1991}.
		Therefore the proof that $S(u)$ is Fr\'echet differentiable can be done 
		by proving that $c$ is differentiable with respect to $U$ and $u$ and 
		the derivative with respect to $U$ is injective. This derivative $\partial_U c$ applied 
		to an element $v \in L^2(\Omega)$ is given by 
		$$
		\partial_U c(U,u) \, v := v - \int_\Omega k(x,y) \,  \Phi(U^+,u)(y) \, \big(1 - \Phi(U^+,u)(y)   \big) \, v(y) \, dy,
		$$
		since we have
		$$
		\partial_U \Phi(U,u)  = - \Phi(U,u) +  \Phi(U,u) ^2;
		$$
		similarly for $\partial_u \Phi(U,u)$.
		Thus, if the equation $\partial_U c(U,u) \, v =0$  
		admits the unique solution $v=0$, for fixed $U(x)$ and $u(x)$, the derivative is invertible. 
		\noindent
		A sufficient condition for the invertibility of the derivative is obtained 
		by the Poincar\'e inequality $\| v  \|_{L^2}^2\le c_\Omega \, \| \nabla v \|_{L^2}^2$.
		Applying the Laplacian operator to the equation  $\partial_U c(U,u) \, v =0$, we get
		\begin{equation}
			-\Delta v(x)=\Phi(U^+(x),u(x))\Big(1-\Phi(U^+(x),u(x))\Big)v(x).
			\label{eqforv}
		\end{equation}
		Let us define $\Psi(x):=\Phi(U^+(x),u(x))\Big(\Phi(U^+(x),u(x))-1\Big)$. Notice that $\Psi(x)\ge -\frac{1}{4}$, therefore if $v$ is a solution
		of \eqref{eqforv} with $v|_{\partial \Omega}=0$, we have
		$$
		-\Big(\nabla v,\nabla v\Big)_{L^2}=\Big(\Psi v,v\Big)_{L^2}\ge -\frac{1}{4}||v||_{L^2}^2.
		$$
		We see that if $c_\Omega<4$ the last two inequalities are in contradiction for any $v\ne 0$, therefore the theorem is proved.
	\end{proof}
	Notice that in our setting the constant $c_\Omega$ corresponds 
	to the inverse of the smallest eigenvalue of the negative Laplacian 
	with homogeneous Dirichlet boundary conditions. 
}

Now, using the map $S$, we can define the reduced cost functional, 
i.e. $\hJ(u):=J(\Phi(S(u),u),u)$. For the cases 1. and 2. (i.e. the $H^1_0(\Omega)$ setting) we have
\begin{equation}
	\hJ (u )= \int_\Omega V(x) \, \Phi(S(u)(x),u(x)) \, dx +   
	\frac{\alpha}{2} \, \| u \|^2_{H^1(\Omega)}, 
	\label{eJ04_1}
\end{equation}
while for the case 3. (i.e. the $H^2_0(\Omega)$ setting) we get
\begin{equation}
	\hJ (u )= \int_\Omega V(x) \, \Phi(S(u)(x),u(x)) \, dx +   
		\frac{\alpha}{2} \, \| u \|^2_{H^2(\Omega)},
	\label{eJ04_2}
\end{equation}

Therefore  the problem of minimizing \eqref{eJ02} subject to the 
differential constraint in \eqref{contrpr} is equivalent to the 
following minimization problem
\begin{equation}
	\min_{u \in D (\Omega)} \hJ (u ).
	\label{eJ04min}
\end{equation}

Now, we prove the following theorem. 
\begin{theorem}
	The reduced cost functional $\hJ$ has a minimum.
\end{theorem} 
\begin{proof} Let  $\underline V:=\inf_{x \in \Omega} V (x )$. We have
\begin{eqnarray*}
	J(\Phi(S(u),u),u)\ge \underline V \int_\Omega  \, \Phi(S(u)(x),u(x)) \, dx = \underline V,
\end{eqnarray*}
thus, the functional $\hJ(u)$ is bounded from below.

Let $\{v_k\}_{k=1}^\infty$ be a minimizing sequence, that is, a sequence such that $\lim_{k\rightarrow \infty}\hat{J}(v_k)=\inf_{u}\hat{J}(u)$. It is not restrictive to assume that
\begin{eqnarray}
\hat{J}(v_k)\le \hat{J}(0),
\label{bdd1}
\end{eqnarray}
{where $\hat{J}(0)$ represents $\hat{J}$ computed at $u(x)\equiv 0$}.
From now on the proof is different in the cases
1. and 3. with respect to the case 2.

{\bf Cases 1. and 3.}
Let us define the function
\begin{eqnarray*}
	m(n)=\left\{\begin{array}{l} 
		1 \quad{\rm if}\,\, n=1,   \\
		2  \quad{\rm if}\,\, n=2, \,\, {\rm or}\,\, n=3.
	\end{array}	
	\right.
\end{eqnarray*}
The bound \eqref{bdd1} implies that the sequences $\{D^i v_k\}_{k=1}^\infty,\,\,i=1,\dots,m(n),$ are bounded. 

Now, since $m(n)p>n$, with $p=2$, by the Rellich-Kondrachov's theorem \cite[Theorem 6.3]{Adams2003}, $H^{m(n)}_0(\Omega)$ is compactly embedded into the space $C^{0,\gamma(n)}(\bar\Omega)$ of the H\"older continuous functions with exponent
$\gamma$, 
where $\gamma(1)=\gamma(3)$ can be any number less than $\frac{1}{2}$ and $\gamma(2)$ can be any number less than $1$.
{Then, there exist a subsequence of $\{v_k\}_{k=1}^\infty$, that we denote again $\{v_k\}_{k=1}^\infty$, with abuse of notation, and a function $v\in C^{0,\gamma(n)}(\bar\Omega)$ such that $v_{k}\rightarrow v$ in $C^{0,\gamma(n)}(\bar\Omega)$.
We also have 
\begin{eqnarray*}
  \int_\Omega V(x) \, \Phi(S(v_k)(x),v_k(x))\stackrel{k\rightarrow \infty}{\rightarrow} \int_\Omega V(x) \, \Phi(S(v)(x),v(x)) \, dx.
\end{eqnarray*}
This result, together with the lower semi-continuity of the $L^2$, according to which
\begin{eqnarray}
	&&||v||_{H^{m(n)}(\Omega)}\le \liminf_{k\rightarrow \infty}||v_k||_{H^{m(n)}(\Omega)},
	\label{lscF}
\end{eqnarray}
shows that $v$ is a point of minimum of the reduced cost functional.}

{\bf Case 2.} By the Rellich-Kondrachov's theorem, there exists $v\in H^1_0(\Omega)$ such that after extraction of a subsequence, if necessary, it holds: 
 \begin{eqnarray*}
 	&&\nabla v_k\rightharpoonup \nabla v, \quad{\rm in} \quad L^2(\Omega),\\
 	&&v_k(x)\rightarrow v(x),\quad a.e.
 \end{eqnarray*}  
Moreover,  {$S$ being continuous}, $S(v_k)(x)\rightarrow S(v)(x)$,  $a.e.$. The functions $S(v_k)(x)$ are solutions of Equation \eqref{eEqU01} with $u=v_k$, therefore the
sequence $\{S(v_k)\}_{k=1}^{\infty}$, as seen before, is bounded in $H^2(\Omega)$, and by the Rellich-Kondrachov's theorem, we can conclude that $S(v_{k})\rightarrow S(v)$ uniformly, after extraction of a subsequence, if necessary.

We also have that
$M_1\le v_k(x)\le M_2$, ${\textrm {a.e.}}$, and $	\exp{(-v_k(x))}\le \exp{(-M_1)}$, ${\textrm {a.e.}}$, therefore
$$
\exp{(-v_k(x)-S(v_k)(x))}\le \exp{(-v_k(x))} \le \exp{(-M_1)},\, {\textrm {a.e.}}
$$
and by the Lebesgue's dominated convergence theorem, we obtain 
$$
\int_\Omega \exp{(-v_k(x)-S(v_k)(x))}\,dx \rightarrow \int_\Omega \exp{(-v(x)-S(v)(x))}\,dx. 
$$
By a similar argument, we have
$$
\int_\Omega V(x)\exp{(-v_k(x)-S(v_k)(x))}\,dx \rightarrow \int_\Omega  V(x)\exp{(-v(x)-S(v)(x))}\,dx. 
$$
The last two results together with the lower semi-continuity of the $L^2$-norm prove that $v$ is point of minimum of the reduced cost functional.
\end{proof}
We remark that the choice of {the control space, including, if appropriate, the 
boundary conditions, is part of the modeling process, which should correspond to 
specific requirements in applications. In particular,  the choice 
$u \in H_0^1(\Omega)$ corresponds to having the control field generated by particles 
in $\Omega$ possibly outside (e.g. encircling) the plasma.} However, 
other choices are possible that would correspond to different 
conditions for $u$ on the boundary $\partial \Omega$, for example 
in the case of a control field generated outside of the domain. 
 
{
We would like to conclude this section with comments concerning the  
stability of the equilibrium configurations determined by 
our method. Certainly, this is an important issue that deserves a separate future  
investigation. However, based on the references   
\cite{BattMorrisonRein1995,Esenturk2016}, it is plausible to conjecture that the equilibrium configurations designed by our method are linearly stable. We remind that an equilibrium solution is linearly stable if solutions of the 
linearized time-dependent Vlasov-Poisson system remain arbitrarily close to the equilibrium solution in some norm for all times, provided that the initial 
condition is sufficiently close to this solution \cite{BattMorrisonRein1995}. 
In our setting with the external control field, this means considering 
solutions of the VP equation in the form $f(x,v,t)=f_s(x,v) +g(x,v,t)$ and 
analyzing the problem 
 	\begin{align*}
  \partial_t g    + v \cdot \nabla_x g 
  - \nabla_x (U  + u) \cdot \nabla_v g &=  \nabla_x U_g  \cdot \nabla_v f_s  , \\
  -\Delta U_g &= \int_{\R^n} g(x,v,t) dv  ,
 	\end{align*}
 	where the stationary density  
 	$f_s$ is given by \eqref{eDensStat01} with $\rho=\Phi(U,u)$, and the 
 	initial condition $g(x,v,0)=g_0(x,v)$ is assumed to be sufficiently small. 
}
{We also remark that Maxwellian distributions with suitable $\rho$ and $T$, as those we treat, are stationary solutions also in the case when collisions are taken into account, as it has to be done when considering sufficiently large time scales. 
In this case, we refer to studies about the asymptotic behaviour of the Vlasov-Poisson-Boltzmann equation, showing convergence to equilibrium solutions \cite{DesvillettesDolbeault1991,DesvillettesVillani2005,Li2008}}.

{\section{Optimality conditions}
\label{sec-optsys}
In this section, we discuss the characterization of optimal design 
functions by first-order optimality conditions based on the 
derivative of $\hJ(u)$ with respect to $u$. Since $\hat{J}(u)=J(\Phi(S(u),u),u)$, differentiability of $\hat J$ requires differentiability 
of $\Phi$, $S$ and $J$ with respect to their respective arguments. 
The differentiability of $\Phi$ and $S$ has been discussed in the previous section; see 
Theorem \ref{teo1}. Further, differentiability of $J$ follows from 
the differentiability of $\Phi$ and the fact that the 
control costs are given by differentiable norms.
}

{
We have already seen that the control-to-state map is implicitly 
defined by the equation $c(U,u)=0$ in the sense that $c(S(u),u)=0$ 
(see Theorem \ref{teo1}). Therefore 
$$
\partial_u S(u)=- \left(  \partial_U \, c(S(u),u) \right)^{-1} \, \partial_u \, c (S(u),u) .
$$
In order to compute the derivative of $\hJ(u)$, we discuss the following 
lemma. 
\begin{lemma}
	Let $U\in H^2(\Omega) \cap H^1_0(\Omega)$ and $u \in D(\Omega)$; 
	assume that the measure of the domain 
	$\mu(\Omega) \ge 1$ but sufficiently small such that the condition 
	$c_\Omega < 4$ is satisfied. Then the linear elliptic equation 
	\begin{equation}
		- \Delta p - \partial_U \Phi(U,u) \, p = - V \,   \partial_U \Phi(U,u) \qquad  \mbox{ in } \Omega, \qquad 
		p=0 \mbox{ on } \partial \Omega  \label{eAdjL} 
	\end{equation}
	has a unique solution $p \in  H^2(\Omega) \cap H^1_0(\Omega)$.
	\label{lemmaadj}
\end{lemma}
\begin{proof}
	Notice that in the case of regular convex domains $\Omega$, we have that 
	$c_\Omega \approx C \, \mu(\Omega)$. In particular, for a 2D square domain 
	of length $L$, we have $\mu(\Omega) = L^2$ and $c_\Omega = \frac{1}{2 \pi^2} \, \mu(\Omega)$. Therefore the assumption of the theorem for this case 
	corresponds to the requirement $1 \le \mu(\Omega) < 8\pi^2$. A similar bound results in the case of a domain being a disc of radius $L/2$. 
	Now, recall the result $\partial_U \Phi(U,u)  = - \Phi(U,u) +  \Phi(U,u) ^2$. It 
	follows that $\partial_U \Phi(U,u)  \le 0$ if $\Phi(U,u) \le 1$, which holds 
	if $\mu(\Omega) \ge 1$. The proof of this claim is immediate 
	considering \eqref{eRHO02} and requiring $\Phi(U,u) \le 1$. Then,  
	integration in $\Omega$ leads to the result. 
	Next, notice that if $\partial_U \Phi(U,u)  \le 0$, then the differential operator 
	in \eqref{eAdjL} is monotone, and the statement of the theorem follows 
	by standard arguments. 
\end{proof}
}

{
Now, we discuss the characterization of an optimal design by means 
of an optimality system. 
\begin{theorem}
	Assume that the conditions of Theorem \ref{teo1} and of 
	Lemma \ref{lemmaadj} are satisfied. 
	Let $(U,u) \in H^2(\Omega) \cap H^1_0(\Omega) \times D(\Omega)$
	be a solution to the optimal control problem: 
	\begin{align}
		\min J(U,u) &:=  \int_\Omega V(x) \, \Phi(U,u)(x) \, dx +   \frac{\alpha}{2} \, \| u \|^2_{\cH} \nonumber \\
		\mbox{ s.t.}  \qquad - \Delta U &= \Phi(U,u) \qquad \mbox{ in } \Omega, \label{contrprx} \\
		U &=0  \qquad \qquad \mbox{ on } \partial \Omega,  \nonumber
	\end{align}
	where the Hilbert spaces $\cH=H^1(\Omega)$ (case 1. \& 2.) or $\cH=H^2(\Omega)$ (case 3.), are endowed with the respective standard scalar product 
	denoted with $(\cdot , \cdot )_\cH$. 
	Then there exists $p \in H^2(\Omega) \cap H^1_0(\Omega)$, such that 
	the triple $(U,u,p)$ satisfies the following optimality system:
	\begin{align}
		- \Delta U-\Phi(U,u) &= 0, & \mbox{ in } \Omega, \quad U=0 \text{ on } \partial \Omega, \nonumber\\
		- \Delta p - \partial_U \Phi(U,u) \, p &= - V \,   \partial_U \Phi(U,u) & \mbox{ in } \Omega, \quad p=0 \text{ on } \partial \Omega,\label{system1}\\
		\big( \mu +\alpha \, u  \, , \,  v-u \big)_{\cH}\, &\geq\, 0, \qquad v \in D(\Omega),
		\nonumber
	\end{align}
	where $\mu$ is the $\cH$-Riesz representative of the continuous linear functional
	\begin{equation}
		v\;\mapsto\; \left(  (  V - p ) \, \partial_u \Phi (S(u),u) \, ,\;v\right)_{L^2(\Omega)} .
		\label{vmapH}
	\end{equation}
\end{theorem} 
}
{
\begin{proof}
	The proof proceeds along the following lines. We determine the derivative 
	of the reduced cost functional with respect to $u$, and its Riesz 
	representative in the control space $\cH$. We denote the 
	resulting (reduced) gradient of $\hJ(u)$ in the space $\cH$ with $\nabla_u \hJ(u)$. 
	Since $u$ is a local minimizer, it must satisfy the following 
	optimality condition
	\begin{equation}
		( \nabla_u \hJ(u), v-u)_{\cH} \ge 0, \qquad v \in  D(\Omega).
		\label{eIneqOpt}
	\end{equation}
	Then, we prove that this condition is equivalent to \eqref{system1}.
	A direct way to determine the reduced gradient is to exploit the fact 
	that $\hJ$ is Fr\'echet differentiable, and therefore we can equivalently 
	determine the G\^ateaux derivative of $\hJ$ with respect to the variation 
	$\delta u$ at $u$ as the following limit (for simplicity, whenever possible, 
	we omit to write the integration variable in the integrals) 
	\begin{align*}
		(\nabla \hat{J}(u), \delta u )_\cH &=\lim\limits_{t \to 0^+} \frac{1}{t} \left(\hat{J}(u+t \, \delta u)-\hat{J}(u) \right)  \\
		&=\lim\limits_{t \to 0^+} \frac{1}{t} \Big(
		\int_\Omega V \, \Phi(S(u+t \, \delta u),u+t \, \delta u) \, dx +   \frac{\alpha}{2} \, \| u +t \, \delta u\|^2_{\cH} \\
		& \qquad -\int_\Omega V \, \Phi(S(u),u) \, dx -  \frac{\alpha}{2} \, \| u \|^2_{\cH} \Big) \\
		&=\int_\Omega V \, \Big( \partial_U \Phi(S(u),u) \, \delta U  
		+ \partial_u \Phi(S(u),u) \,  \delta u  \Big) \, dx +   \alpha\, (u , \delta u)_{\cH} 
	\end{align*}
	In this calculation, we have used the approximation (valid in the limit ${t \to 0^+} $)
	$S(u+t \, \delta u) = S(u) + t \, \partial_u S(u) \, \delta u$, and defined 
	$\delta U = \partial_u S(u) \, \delta u$. Notice that the pair 
	$(\delta U , \, \delta u)$ satisfies the linearized constraint equation
	$$
	-\Delta \delta U -  \partial_U \Phi(U,u) \, \delta U = \partial_u \Phi(U,u) \, \delta u,
	$$
	with homogeneous Dirichlet boundary conditions. We use this equation and 
	the so-called adjoint equation given in Lemma \ref{lemmaadj} to 
	continue the calculation above as follows. 
	\begin{align*}
		&	(\nabla \hat{J}(u), \delta u )_\cH \\
		&=\int_\Omega V \, \Big( \partial_U \Phi(S(u),u) \, \delta U  
		+ \partial_u \Phi(S(u),u) \,  \delta u  \Big) \, dx +   \alpha\, (u , \delta u)_{\cH} \\
		& = \int_\Omega \Big[\Big(  \Delta p + \partial_U \Phi(S(u),u) \, p \Big) \, \delta U  
		+ V \,  \partial_u \Phi(S(u),u) \,  \delta u  \Big] \, dx +   \alpha\, (u , \delta u)_{\cH} \\
		& = \int_\Omega \Big[\Big(  \Delta \delta U  + \partial_U \Phi(S(u),u) \, \delta U  \Big) \, p + V \,  \partial_u \Phi(S(u),u) \,  \delta u  \Big] \, dx +   \alpha\, (u , \delta u)_{\cH} \\
		& = \int_\Omega \Big[ - p \, \partial_u \Phi(S(u),u) \, \delta u + V \,  \partial_u \Phi(S(u),u) \,  \delta u  \Big] \, dx +   \alpha\, (u , \delta u)_{\cH} \\
		& = \int_\Omega \Big[ (V- p) \, \partial_u \Phi(S(u),u) \, \delta u  \Big] \, dx +   \alpha\, (u , \delta u)_{\cH} \\
		& = ( \mu + \alpha \, u , \delta u)_{\cH},
	\end{align*}
{from which we get
\begin{equation}
\label{eq:H^1-gradient}
	\nabla_u\hat J(u)=\mu+\alpha u.
\end{equation}}
	Thus the theorem is proved. 
\end{proof}
}

{
We remark that in the case $\cH=H^1(\Omega)$, 
the $H^1$-Riesz representative of \eqref{vmapH} can be obtained by solving the 
following boundary-value problem 
\begin{equation}
	\label{ellEqmu2xx}
	\left(-  \,\Delta +  I\right) \, \mu =   (  V - p ) \, \partial_u \Phi  \qquad
	\mbox{ in } \Omega, \qquad 
	\mu=0 \mbox{ on } \partial \Omega,
\end{equation}
which has to be understood in a weak sense; {$I$ represents the identity operator}. 
On the other hand, if $\cH=H^2(\Omega)$, then the $H^2$-Riesz representative 
of \eqref{vmapH} is given by the weak solution to the following problem
\begin{eqnarray*}
&\big(\Delta^2 -\Delta + I\big) \, \mu = \big(V -p\big)\partial_u \Phi(U,u) \quad {\textrm in}\quad \Omega,\\
&{\mu=0, \quad \nabla \mu=0,\quad {\textrm on}\quad \partial\Omega}.
\end{eqnarray*}
}

{
In the next Section, we discuss a numerical iterative procedure that 
aims at solving \eqref{system1} by constructing a minimizing 
sequence of the cost functional in cases 1 and 2. In order to provide a rigorous foundation of this procedure, in the following 
we 	{prove that the system we will have to solve in the procedure has solutions}.  
\begin{theorem}\label{teoiter}
	{Let $u\in D(\Omega)$ be fixed.} The sequence $(U_k,p_k,\mu_k) \subset  
	\big(H^2(\Omega) \cap H^1_0(\Omega) \big)^3$ defined by the 
	iterative procedure 
	\begin{align}
		- \Delta U_k&=\Phi(U_{k-1},u) ,\quad  \mbox{ in }\quad \Omega,  \nonumber\\
		- \Delta p_k - \partial_U \Phi(U_{k-1},u) \, p_k &= - V \,   \partial_U \Phi(U_{k-1},u)  \mbox{ in } \Omega ,\label{system2a}\\
		- \Delta \mu_k +\mu_k  &=(V-p_{k-1}) \,   \partial_u {\Phi(U_{k-1},u)}  \mbox{ in } \Omega, \nonumber
	\end{align}
	with starting values $U_0=p_0=\mu_0\equiv 0$ converges to a solution of 
	{
		\begin{align}
		- \Delta U-\Phi(U,u) &= 0, & \mbox{ in }\quad \Omega, \quad U=0 \text{ on } \quad\partial \Omega, \nonumber\\
		- \Delta p - \partial_U \Phi(U,u) \, p &= - V \,   \partial_U \Phi(U,u) & \mbox{ in } \Omega, \quad p=0 \text{ on } \quad\partial \Omega,\label{system2}\\
		- \Delta \mu +\mu  &=( V-p) \,   \partial_u \Phi(U,u) & \mbox{ in } \Omega, \quad \mu=0 \text{ on } \quad\partial \Omega . \nonumber
	\end{align}}
	\end{theorem}
\begin{proof}
	Notice that each equations in \eqref{system2} is a well-posed 
	linear elliptic equation with homogeneous Dirichlet boundary conditions. 
	Now, the system \eqref{system2} defines a procedure that constructs 
	bounded sequences of functions $U_k$, $p_k$ and $\mu_k$. 
	Then, there exists a subsequence of $\{U_k\}_{k=0}^\infty$ which converges to a function $\overline U\in H^2(\Omega)$, weakly in $H^2(\Omega)$, and uniformly in $C(\Omega)$, see the sketch of the proof of the existence theorem for equation \eqref{eEqU01}. Analogous results hold for $\{p_k\}_{k=0}^\infty$ and
	$\{\mu_k\}_{k=0}^\infty$.  In fact, we have
	$$
	\vert V(x)\Phi_U(U,u)\vert\le \overline V(\Phi+\Phi^2),\,\,{\textrm {with}}\,\, \overline V=\sup_{x\in \Omega }\vert V(x)\vert 
	$$
	and 
	$$
	\int_\Omega\Phi(x)^{2p}\,dx=\vert\vert \Phi\vert\vert^{2p}_{L^{2p}(\Omega)}\le C_p^{2p}\vert\vert \Phi\vert\vert^{2p}_{L^{1}(\Omega)}= C_p^{2p},
	$$
	with $1<p<\infty$, and $C_p$ positive constants.
	From the previous two inequalities it follows that the $L^2$ norm of the right-hand side of \eqref{system2}{$_2$} has a bound which is independent of $k$.
	Therefore, we can state that there exists $C'>0$ such that
	$$
	\vert\vert p_k\vert\vert_{H^2(\Omega)}<C',\,\,\forall k\in\mathbb{N} .
	$$
	Consequently there exists a subsequence of $\{p_k\}_{k=0}^\infty$ which converges to a function $\overline p\in H^2(\Omega)$, weakly in $H^2(\Omega)$, and uniformly in $C(\Omega)$.
	Analogously, it can be proved that the same holds also for the sequence $\{\mu_k\}_{k=0}^\infty$. In conclusion, passing to the limit for 
	$k\rightarrow +\infty$ in \eqref{system2}, we find the solution of system \eqref{system2}.
\end{proof}
}

\section{Numerical approximation and optimization schemes}
\label{sec-gradient}

The numerical computation of an optimal control field requires 
the approximation of the {optimality system} \eqref{system2} in the assembling of the reduced gradient. 

We consider the Cases 1. and 2. in which the gradient is given by $\nabla_{u} \hat J(u)= \alpha u+\mu$, see formula \eqref{eq:H^1-gradient}. At each step of 
the gradient-based optimization procedure discussed below, we start from $u_k$ and seek $U_{k+1}$ that solves the following boundary value problem  
\begin{eqnarray}
	- \Delta U_{k+1}=\Phi(U_{k+1},u_k) ,\quad  \mbox{ in }\quad \Omega, \quad U_{k+1}=0 \text{ on } \quad\partial \Omega.
\label{equU}
\end{eqnarray}
The second step is to determine $p_{k+1}$ by solving the problem
\begin{eqnarray}
	- \Delta p_{k+1}\! -\! \partial_U\! \Phi(U_{k+1},u_k)  p_{k+1}\! =\! - V  \partial_U \Phi(U_{k+1},u_k)  \mbox{ in } \Omega, \, p_{k+1}=0 \text{ on } \,\partial \Omega.
	\label{equp}
\end{eqnarray}
In the third step, we find $\mu_{k+1}$ by solving the following problem
\begin{eqnarray}
	- \Delta \mu_{k+1} +\mu_{k+1}  &=( V-p_{k+1}) \,   \partial_u \Phi(U_{k+1},u_k)  \mbox{ in } \Omega, \, \mu_{k+1}=0 \text{ on } \,\partial \Omega.
	\label{equmu}
\end{eqnarray}
Hence, we obtain the reduced gradient at $u_k$ given by $\nabla \hat J(u_{k})=\alpha u_k+\mu_{k+1}$.

In order to solve the nonlinear problem \eqref{equU}, we 
apply an iterative fixed-point procedure where each step reads $U_{k+1}^{(0)} \to U_{k+1}^{(1)}$ starting with the initial guess $U_{k+1}^{(0)}=U_k$. We have 
\begin{eqnarray}
	- \Delta U_{k+1}^{(1)}=\Phi(U_{k+1}^{(0)},u_k),\quad  \mbox{ in }\quad \Omega, \quad U_{k+1}^{(1)}=0 \text{ on } \quad\partial \Omega.
	\label{equU1}
\end{eqnarray}
For the approximation of the equations \eqref{equU1}-\eqref{equp}-\eqref{equmu}, we extend to our models the finite-difference framework for Poisson problems developed by A.~S.~Reimer and A.~F.~Cheviakov \cite{RAC13}. 

For the sake of {self-consistency}, we illustrate this framework in the case of Equation \eqref{equU1}.  
{We investigate the spatially two-dimensional case, with the unitary circle as a domain, using the polar coordinates $\varphi$ and $r$.	
In this case equation \eqref{equU} reads
\begin{eqnarray*}
\frac{\partial^2 U_{k+1}}{\partial r^2}+\frac{1}{r}\frac{\partial U_{k+1}}{\partial r}+\frac{1}{r^2}\frac{\partial^2 U_{k+1}}{\partial \varphi^2}=
\Phi(U_{k+1},u_k).
\end{eqnarray*}
Similar expressions hold for \eqref{equp} and \eqref{equmu}. The angular and polar coordinates are discretized as follows
}:  
\begin{eqnarray*}
	\varphi_i&=&\frac{i-1}{N}\,2\pi,\,\,{i=1,\dots, N,}\,\,{\textrm {with}}\,\,N=150,\\
	r_j&=&-2\Bigg[0.5\Big(\frac{j-1}{M-1}\Big)^2-\frac{j-1}{M-1}\Bigg],\,\,{j=1,\dots, M,}\,\,{\textrm {with}}\,\,M=100,
\end{eqnarray*}  	
with the discretization in $r$ being denser near the centre of the circle.

On this grid, the approximate solution consists of the $N\times M$-vector
\[
\left(
	U_{1 1}, \, 
	U_{2 1}, \,
	\ldots\\
	U_{N 1}, \, 
	U_{12} , \; 
	\ldots , 
	U_{N M} \right)^\top,
\]
with $U_{ij}\approx U(\varphi_i,r_j)$, $i=1,\dots, N$, $j=1,\dots, M$. The algebraic system representing the discretized version of \eqref{equU1}  is given by
\begin{equation}
-\sum_{\beta=1}^{N\times M}A_{\alpha\beta} \, U_\beta=f_\alpha,
\label{discr1}
\end{equation}
where $\alpha=(i,j)$, $\beta=(k,l)$, $i,k\in\{1, \dots, N\}$, $j,l\in\{1, \dots, M\}$.
The matrix $A$ is the discretized version of the Laplace operator obtained by using central differences and it takes into account also the homogeneous Dirichlet boundary conditions. 
For $j\not\in \{1, M\}$, the only elements of $A$ different form zero are
\begin{eqnarray*}
	A_{(i,j;i,j)}&=&\frac{2r_j^2}{\Delta r_j+\Delta r_{j-1}}\Bigg(\frac{1}{\Delta r_{j-1}}-\frac{1}{\Delta r_{j}}\Bigg)
	-\frac{2}{\Delta \varphi_i+\Delta \varphi_{i-1}}\Bigg(\frac{1}{\Delta \varphi_{i-1}}+\frac{1}{\Delta \varphi_{i}}\Bigg),\\
	A_{(i,j;i+1,j)}&=&\frac{2}{(\Delta \varphi_i+\Delta \varphi_{i-1})\Delta \varphi_{i}},\\
	A_{(i,j;i-1,j)}&=&\frac{2}{(\Delta \varphi_i+\Delta \varphi_{i-1})\Delta \varphi_{i-1}},\\
	A_{(i,j;i,j+1)}&=&\frac{2r_j^2}{(\Delta r_j+\Delta r_{j-1})\Delta r_{j}}+\frac{r_j}{\Delta r_j+\Delta r_{j-1}},\\
	A_{(i,j;i,j-1)}&=& \frac{2r_j^2}{(\Delta r_j+\Delta r_{j-1})\Delta r_{j-1}}-\frac{r_j}{\Delta r_j+\Delta r_{j-1}},
\end{eqnarray*}
with $\Delta \varphi_i=\varphi_i-\varphi_{i-1},\,i=2,\dots, N$, $\Delta r_j=r_j-r_{j-1},\,j=2,\dots, M$, while
$$f_\alpha=r_j^2
\Phi(U_{k+1}^{(0)}(\varphi_i,r_j),u_k(\varphi_i,r_j)),
\,i=1,\dots, N,\, j\not\in\{1,M\}.$$

At the origin, the polar coordinates have an apparent singularity, which can be treated by means of the approximation 
$$
(\Delta U)(0,0)\approx \frac{1}{|S|}\int_{\partial S}div(\nabla U)\,dl\approx\frac{1}{\pi(\frac{\Delta r_1}{2})^2}\sum_{i=1}^{N}\Big(U_{i2}-U_{i1}\Big)
\Big( \frac{\Delta\varphi_i+\Delta\varphi_{i-1}}{4}\Big),
$$
where $S$ is the circle of radius $\frac{1}{2}\Delta r_1$, $|S|$ its area, $dl=\frac{1}{2} \Delta r_1\Delta \varphi$ and $\Delta \varphi_0=2\pi-\varphi_N$, see \cite{RAC13} for details. Considering this approximation and the condition at the origin: 
$$
U_{11}=U_{21}=\dots=U_{N1},
$$
which must be added since the value of $U$ in the centre is independent of $\varphi$, 
 one can find that
\begin{eqnarray*}
	A_{(1,1;1,1)}&=&-\pi,\quad A_{(1,1;i,2)}= \frac{\Delta\varphi_i+\Delta\varphi_{i-1}}{4},\\
	A_{(i,1;1,1)}&=&-1,\quad A_{(i,1;i,1)}=1,\quad i=2,3,\dots, N,
\end{eqnarray*}
and
\begin{eqnarray*}
 f_{(1,1)}= \pi\Big(\frac{\Delta r_1}{2}\Big)^2\,\Phi(U_{k+1}^{(0)}(0),u_k(0)).
\end{eqnarray*}
Finally, we have
$$
A_{(i,M;i,M)}=1, \quad i=1, 2, \dots, N,
$$
which stem from the Dirichlet homogeneous boundary conditions.
We remark that all other entries of $A$ and $f$ are null.

{As regards the discretized version of \eqref{equp} and \eqref{equmu}, the matrices of elements 
$\partial_U \Phi(U_{k+1}(\varphi_i,r_j),u_{k}(\varphi_i,r_j))\delta_{i,j;k,l}$ and $-\delta_{i,j;k,l}$ have to be added to the above matrix $A$,
where $\delta_{i,j;k,l}$ is equal to $1$ if $i=k$ and $j=l$ and to $0$ otherwise.
Finally the right-had sides are given by the $f$ appearing in \eqref{discr1}, where in its expression the function $\Phi$ has to be replaced with $-V\partial_U \Phi$ and
$(V-p_{k+1})\partial_u \Phi$, respectively. }
\medskip
{Consistent with our approximation of the optimality system is 
the approximation of the functional $J$ by a midpoint/rectangular rule 
in the $\varphi$ coordinate, and a trapezoidal rule in the $r$ coordinate. 
This integration rule also defines our discrete $L^2$ scalar product.}

{We remark that our discretization of the 
elliptic differential operators and of the cost functional makes the 
approach of  optimize-then-discretize equivalent to the discretize-then-optimize strategy.
The former is our approach where we derive the optimality system and 
then discretize the differential equations appearing in it. In 
this procedure, we construct finite-differences approximations of the equations that 
have the same structural properties of their continuous counterpart. 
This can be seen considering the elements of the matrix $A$ given above, 
and the fact that other terms like $\partial_U \Phi$ and $V$ are 
represented and evaluated pointwise on the grid. One can verify that, by first 
discretizing $J$ and the Poisson-Boltzmann equation, and then deriving 
the optimality system with respect to the discrete variables involved, 
the same discretization of the optimality system is obtained.
}

\medskip

Now, we are ready to discuss a gradient-based procedure for solving 
our optimization problem. The basic procedure is steepest descent \cite{BorziSchulz2011}, where starting from an initial guess $u^0$ a sequence of 
approximations $(u^k)$ is obtained as follows
$$
u^{k+1} = u^k - s_k \, \nabla \hJ(u^k), \qquad k=0,1,2, \ldots ,
$$
where $s_k > 0$ is a stepsize given by some criteria or determined 
by a linesearch procedure. Notice that at each step, the gradient 
for the given $u^k$ must be computed. This is done in the 
way illustrated by the following algorithm defined at the 
continuous level. 

\vspace{0.3cm}
{
\begin{algorithm}[H]
\caption{Calculate the gradient $ \nabla \hJ(u) |_{H^1}(x)$}
\label{algo:CalculateGradient}
\begin{algorithmic}[1]
\REQUIRE $u(x)$
\STATE Solve the governing model with given $u(x)$
\STATE Solve adjoint problem with given $u(x)$ and $U(x)$
\STATE Compute $\nabla \hJ(u) |_{H^1}(x) $ 
\end{algorithmic}
\end{algorithm}}

\vspace{0.3cm}

We remark that this algorithm can be used to implement many 
different gradient-based optimization schemes \cite{BorziSchulz2011}. 
In our case, we choose the non-linear conjugate gradient (NCG) method. 
This is an iterative method that constructs a minimizing sequence of 
control functions $(u^k)$ as illustrated by the following algorithm.

\vspace{0.3cm}

\begin{algorithm}[H]
\caption{NCG scheme}
\label{algo:OptAlgo}
\begin{algorithmic}[1]
\REQUIRE $u^0(t)$
\STATE $k = 0$, $E > tol$
\STATE Compute $d^0= - \nabla \hJ(u^0) |_{H^1}$ using Algorithm \ref{algo:CalculateGradient}
\WHILE{$E >tol$ \AND $k < k_{\max}$}
\STATE Use a linesearch scheme to determine the step-size $s_k$ along $d^k$
\STATE Update control: $u^{k+1} = u^k + s_k \, d^k$
\STATE Compute $g^{k+1}= \nabla \hJ(u^{k+1}) |_{H^1}$ using Algorithm \ref{algo:CalculateGradient}
\STATE Compute $\beta_k$ using the Fletcher--Reeves formula
\STATE Set $d^{k+1} = - g^{k+1} + \beta_k \, d^k$
\STATE $E=\|u^{k+1}-u^k\|$
\STATE Set $k = k+1$
\ENDWHILE
\RETURN $(u^k,U^k)$
\end{algorithmic}
\end{algorithm}

\vspace{0.3cm}

In this algorithm, the tolerance $tol>0$ and the maximum 
number of iterations $k_{\max} \in \N$ are used as termination criteria. 
We use backtracking line-search with Armijo condition. The 
factor $\beta_k$ is based on the Fletcher -- Reeves formula:
$$
\beta_k = \frac{(g^{k+1} ,g^{k+1} ) }{(g^{k} ,g^{k} )}; 
$$
see \cite{BorziSchulz2011} for more details, alternative formulas, and references.

\vspace{0.3cm}

\section{Numerical experiments}
\label{sec-numexp}

In this section, we present results of experiments for 
three examples of confinement which differ by the choice of the valley function $V$.  
{This choice is dictated by the purpose of showing that the optimization framework works for valley functions that have increasing geometric complexity.}
We compare the resulting charge distribution with that corresponding to a zero external electric field as depicted in Figure \ref{fig_rho_0}.\\
\begin{figure}[ht!]
	\centering
	\includegraphics[width=0.45\textwidth]{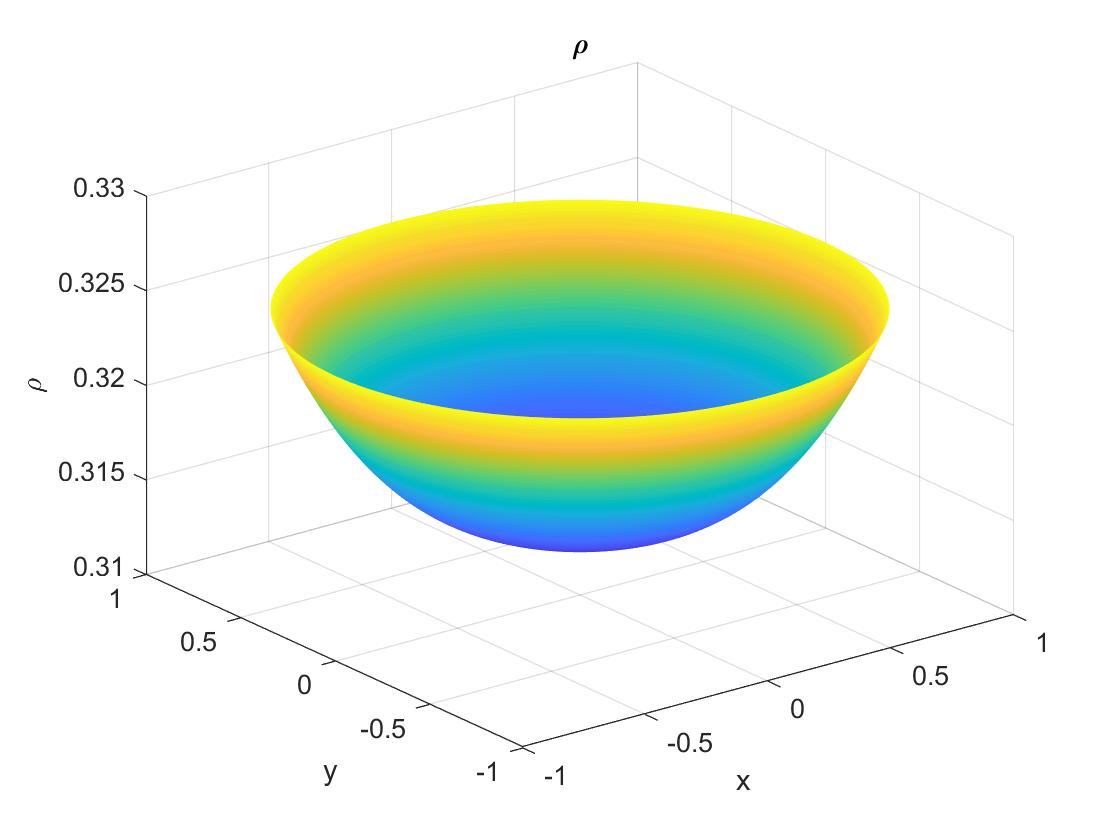}
	\caption{Charge distribution without external field, $u=0$.}
	\label{fig_rho_0}
\end{figure}

In Example 1, we take $V=-\exp{\Big(-\frac{r^2}{2(0.05)^2}\Big)}$. The algorithm converges to a minimum after $100$ iterates. The valley function and the resulting 
plasma distribution are depicted in Figure \ref{example_1_1}. In 
Figure \ref{example_1_2}, we plot the control field $u$ and the 
self-consistent potential $U$. 

\begin{figure}[ht!]
	\centering
	\includegraphics[width=0.45\textwidth]{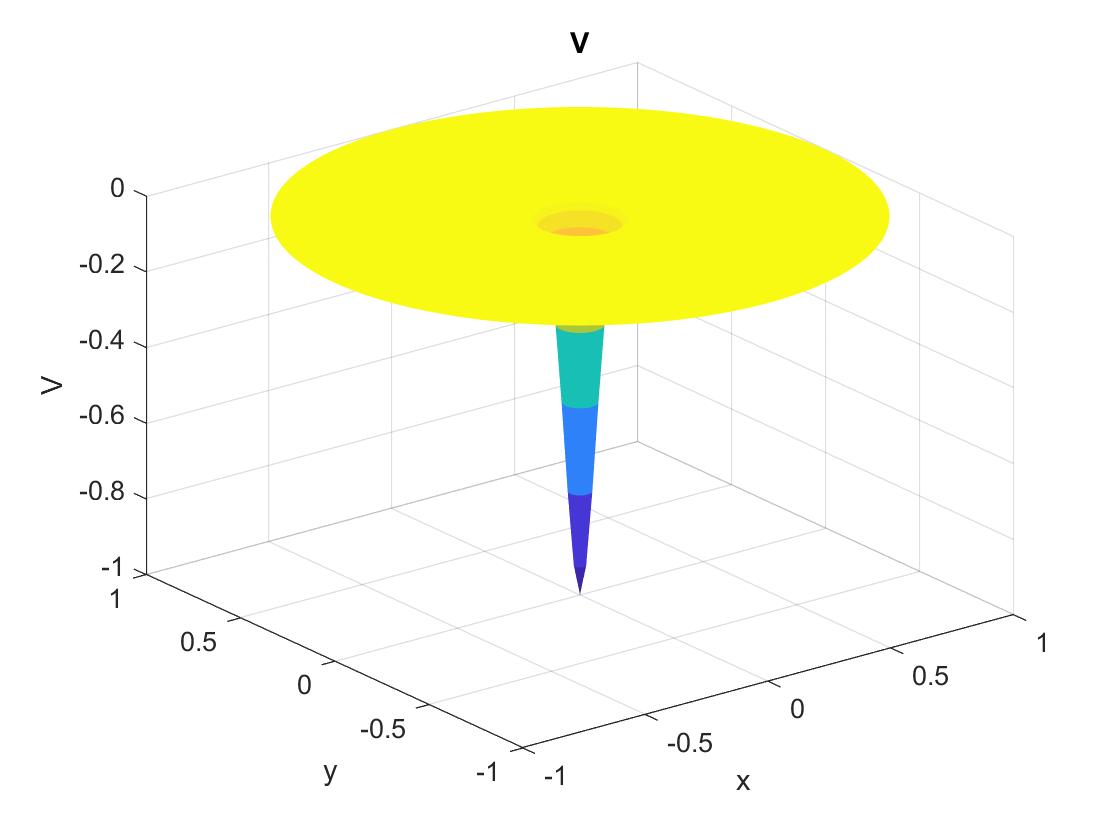}
	\includegraphics[width=0.45\textwidth]{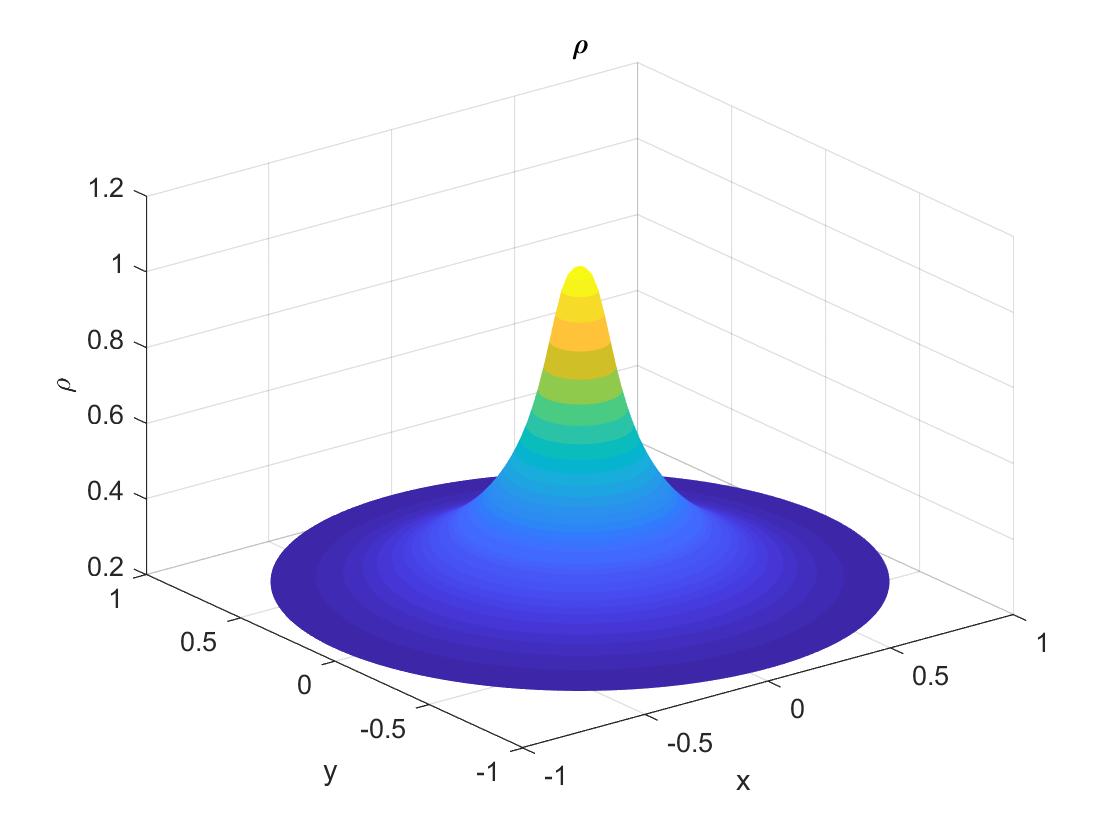}
	\caption{Left: Valley function. Right: Charge distribution. The figure refers to Example 1.}
	\label{example_1_1}
\end{figure}

\begin{figure}[ht!]
	\centering
	\includegraphics[width=0.45\textwidth]{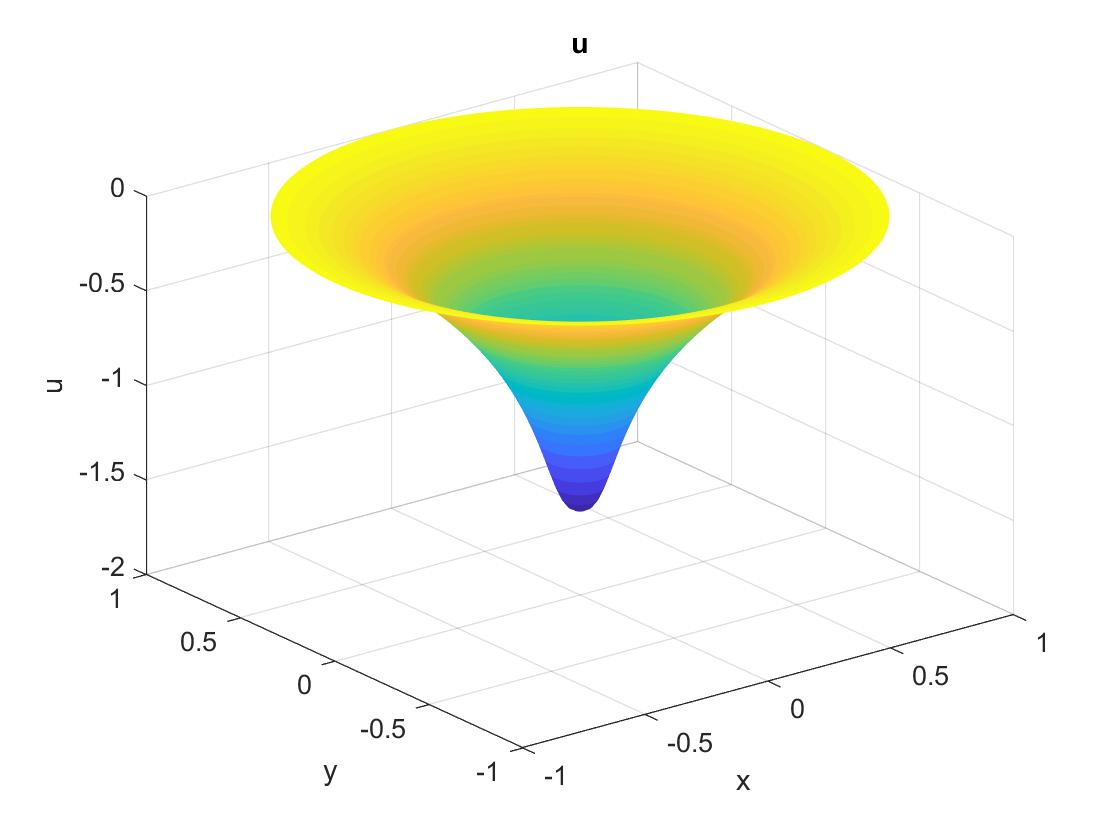}
	\includegraphics[width=0.45\textwidth]{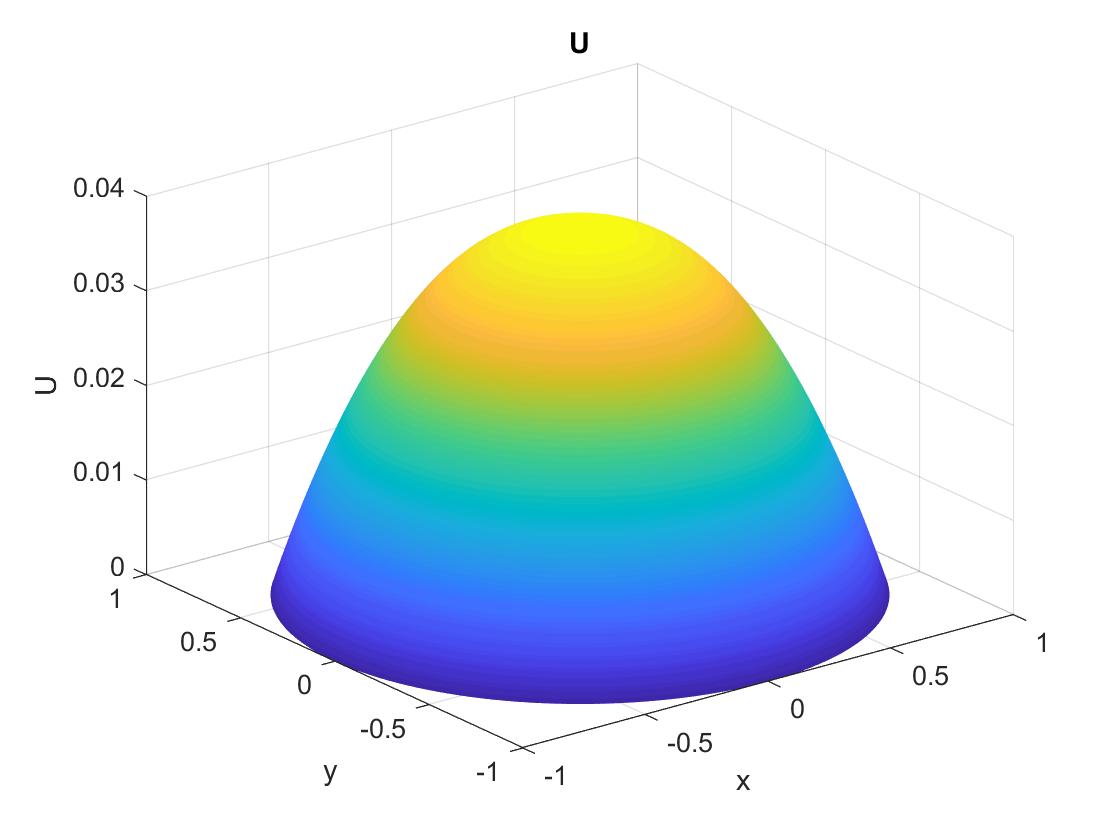}
	\caption{Left: External potential. Right: Self-consistent potential. The figure refers to Example 1.	}
	\label{example_1_2}
\end{figure}

In Example 2, we consider {$V=-\exp{\Big[-r^2\Big(\frac{\cos^2{(\varphi)}}{2(0.05)^2}+\frac{\sin^2{(\varphi)}}{2(0.3)^2}\Big)\Big]}$. The algorithm converges to a minimum after $61$ iterates. The valley function and the resulting 
plasma distribution are depicted in Figure \ref{example_2_1}. In 
Figure \ref{example_2_2}, we plot the control field $u$ and the 
self-consistent potential $U$. We stress that in this strongly anisotropic case the combined effect of the Coulomb repulsion and the action of the control field yields a bivariate charge distribution.}

\begin{figure}[ht!]
	\centering
	\includegraphics[width=0.45\textwidth]{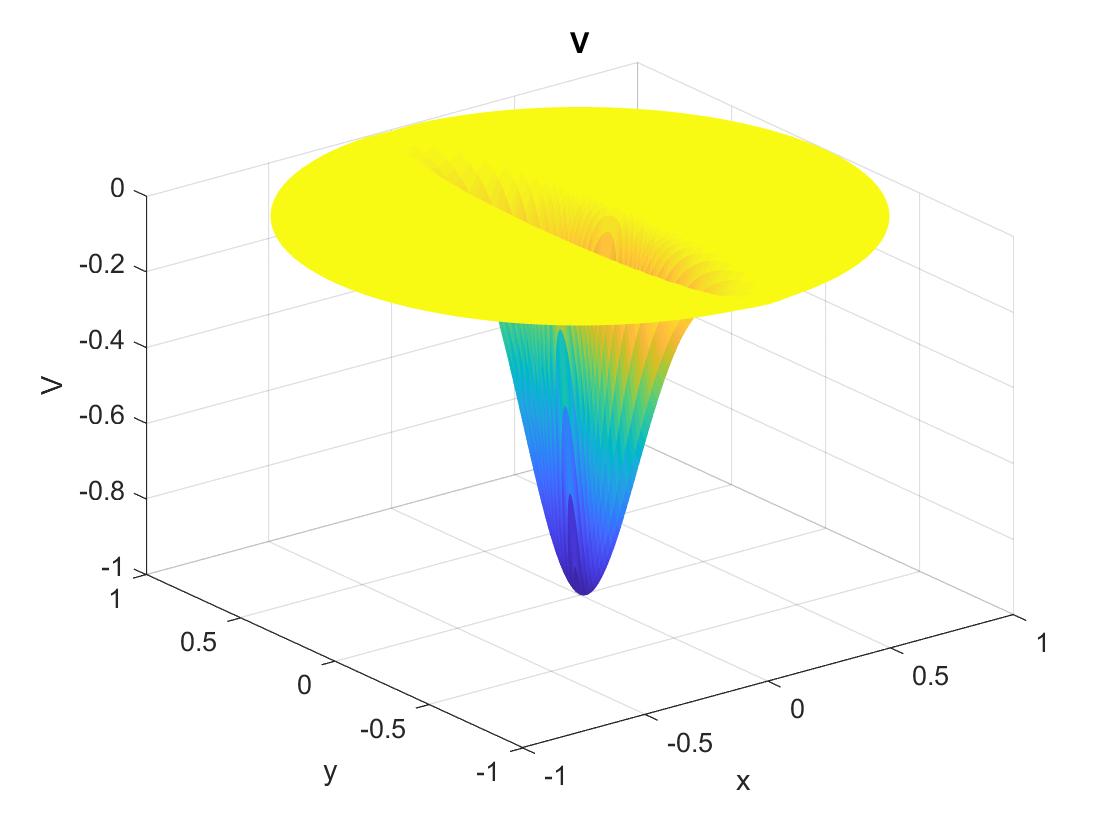}
	\includegraphics[width=0.45\textwidth]{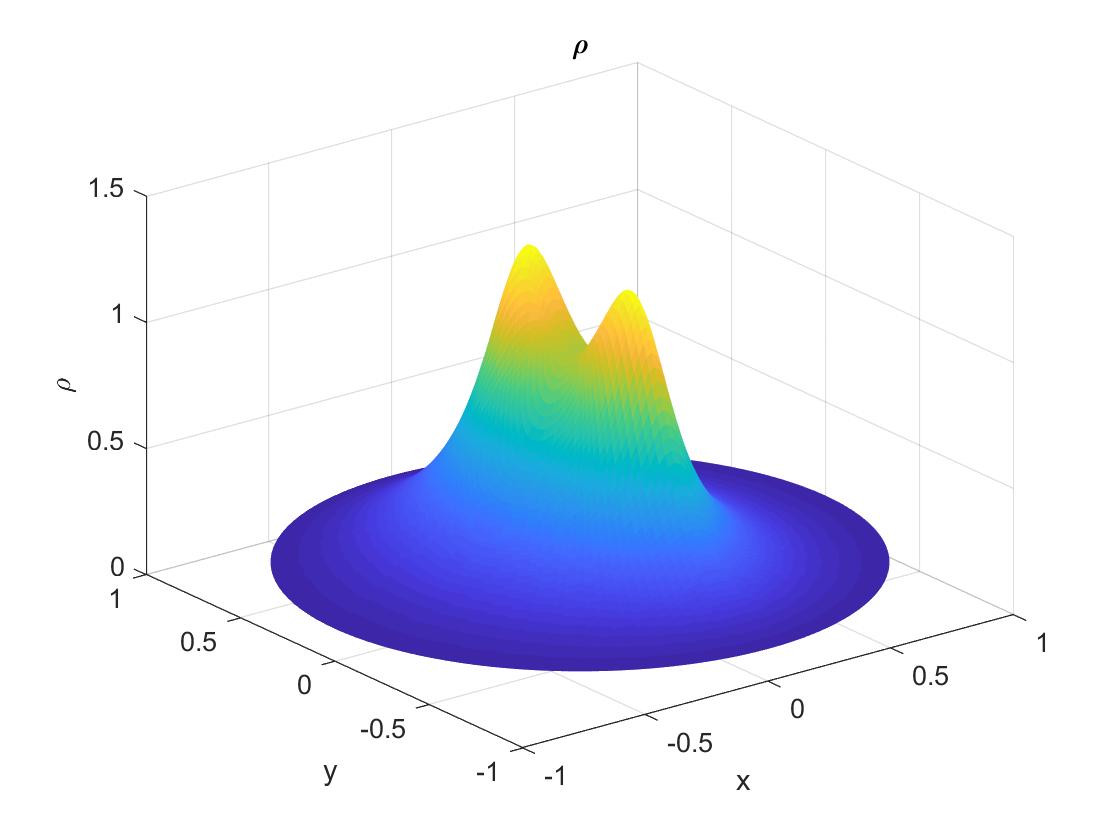}
	\caption{Left: Valley function. Right: Charge distribution. The figure refers to Example 2.	}
	\label{example_2_1}
\end{figure}
\begin{figure}[ht!]
	\centering
	\includegraphics[width=0.45\textwidth]{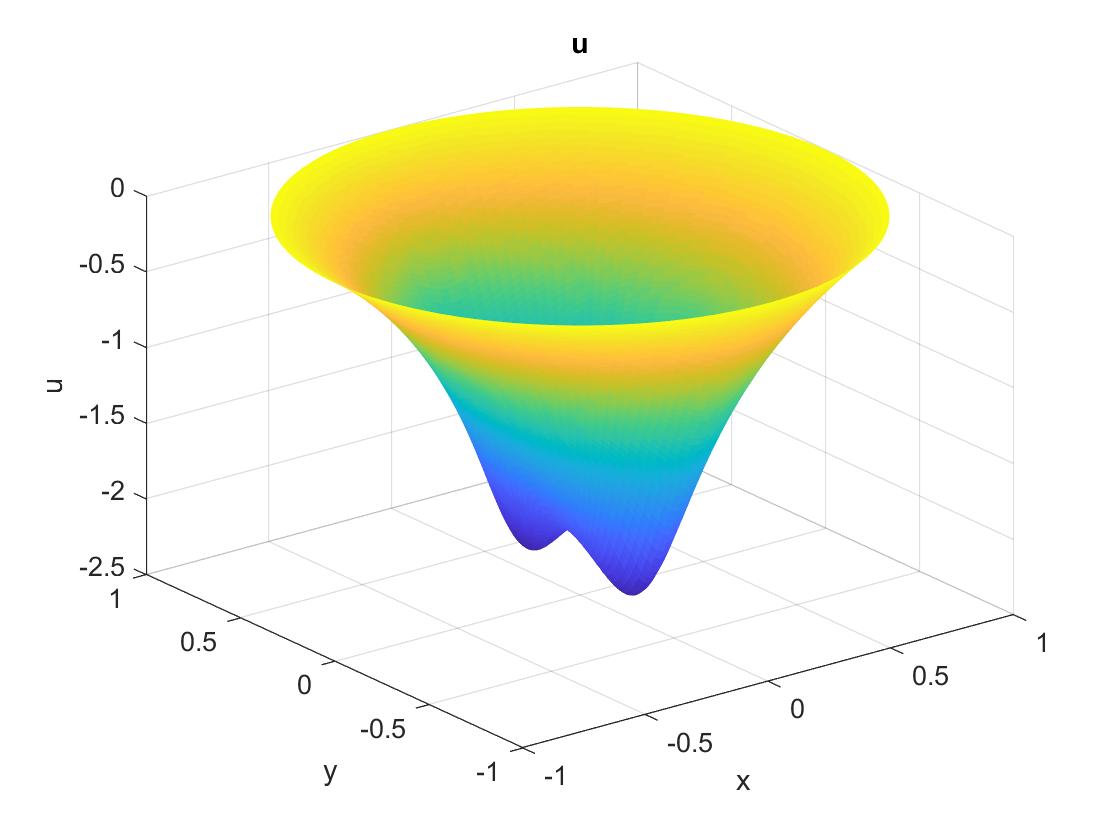}
	\includegraphics[width=0.45\textwidth]{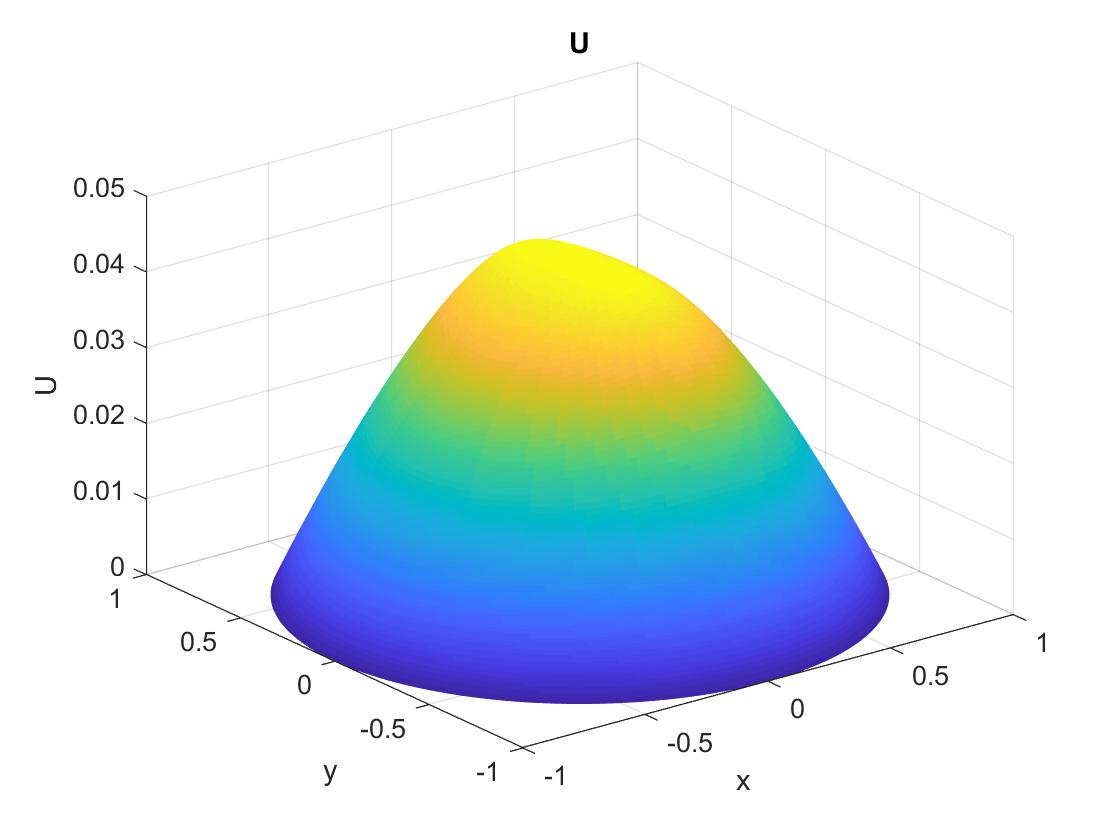}
	\caption{Left: External potential. Right: Self-consistent potential. The figure refers to Example 2.	}
	\label{example_2_2}
\end{figure}

In Example 3, we take a discontinuous valley potential with $V=-1$ inside the four-leaf clover and  $V=0$ outside it, see Fig.\ref{quadr}. The algorithm converges to a minimum after $66$ iterates. The valley function and the resulting 
plasma distribution are depicted in Figure \ref{example_3_1}. In 
Figure \ref{example_3_2}, we plot the control field $u$ and the 
self-consistent potential $U$. 

\begin{figure}[ht!]
	\centering
	\includegraphics[width=0.45\textwidth]{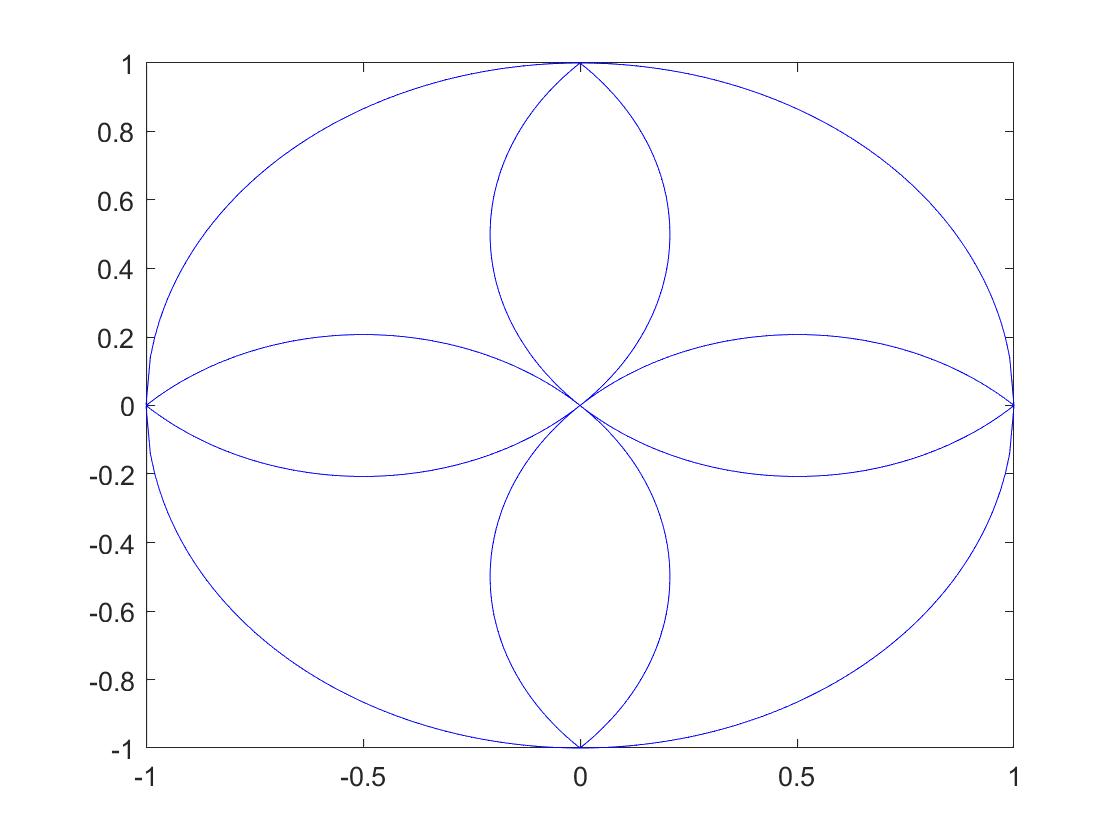}
	\caption{The cloverleaf inside of which $U=-1$ and outside of which $ U=0$.	}
	\label{quadr}
\end{figure}

\begin{figure}[ht!]
	\centering
	\includegraphics[width=0.45\textwidth]{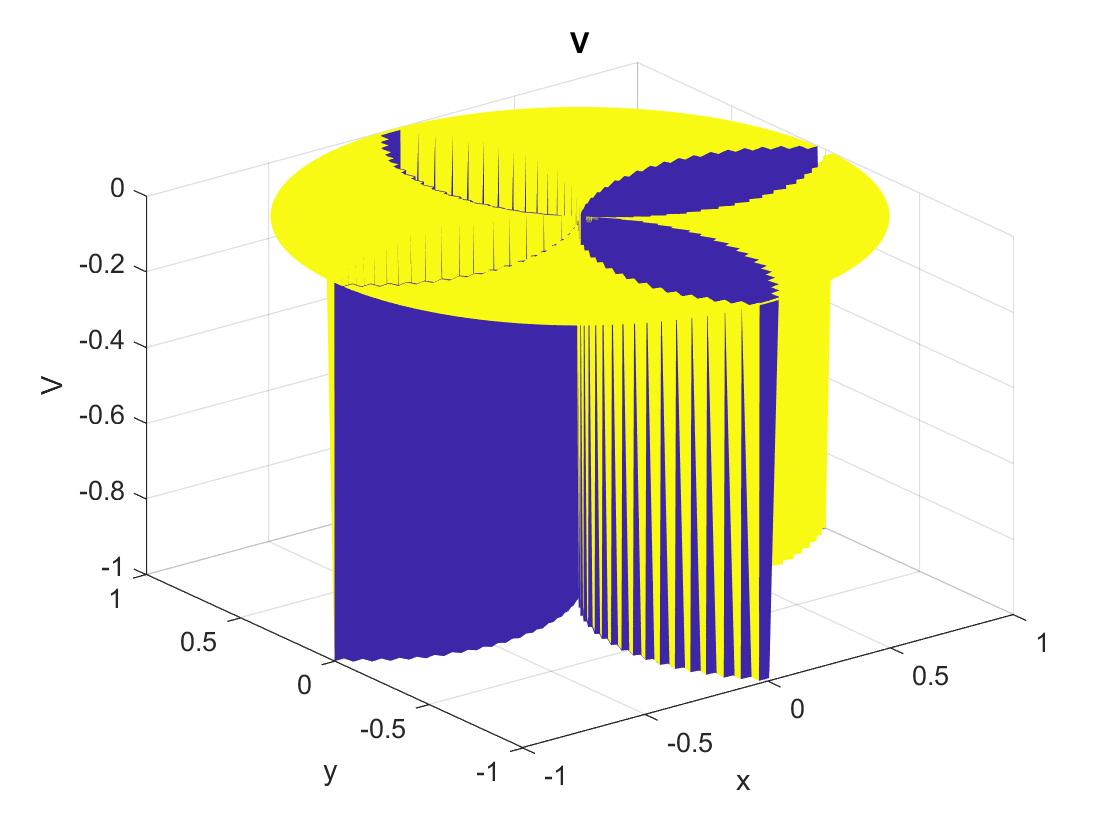}
	\includegraphics[width=0.45\textwidth]{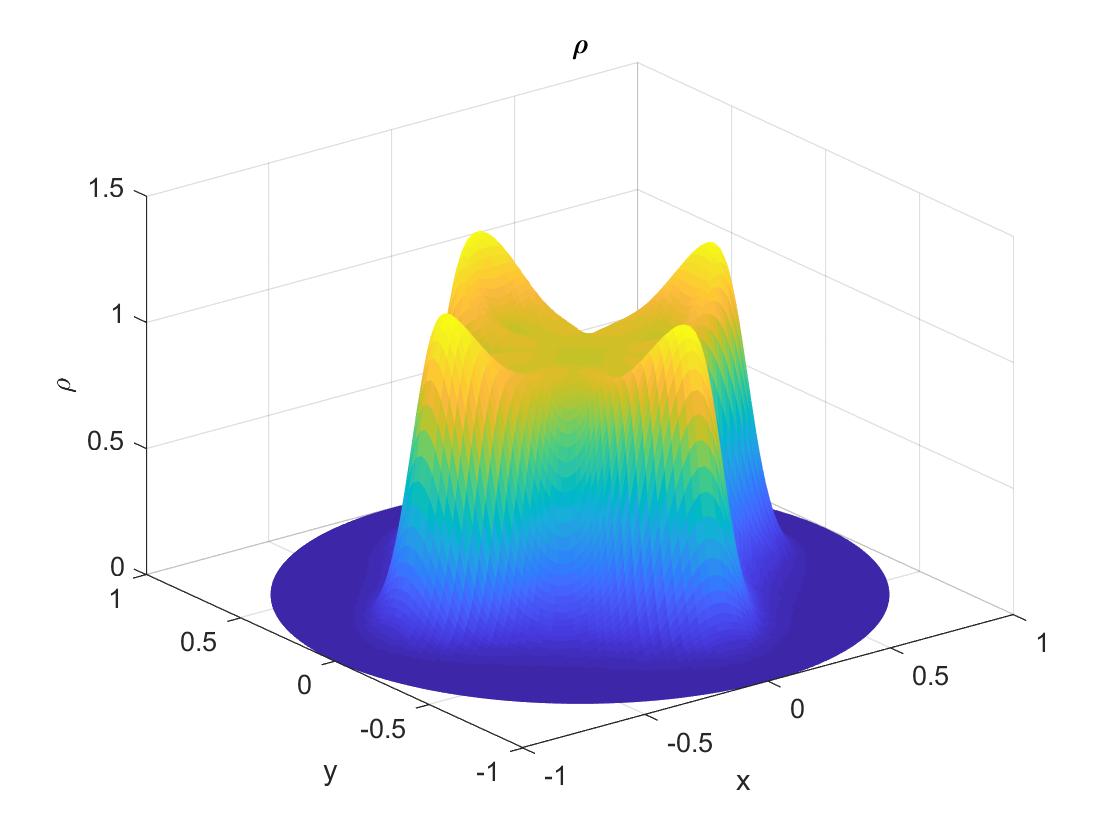}
	\caption{Left: Valley function. Right: Charge distribution. The figure refers to Example 3.	}
	\label{example_3_1}
\end{figure}
\begin{figure}[ht!]
	\centering
	\includegraphics[width=0.45\textwidth]{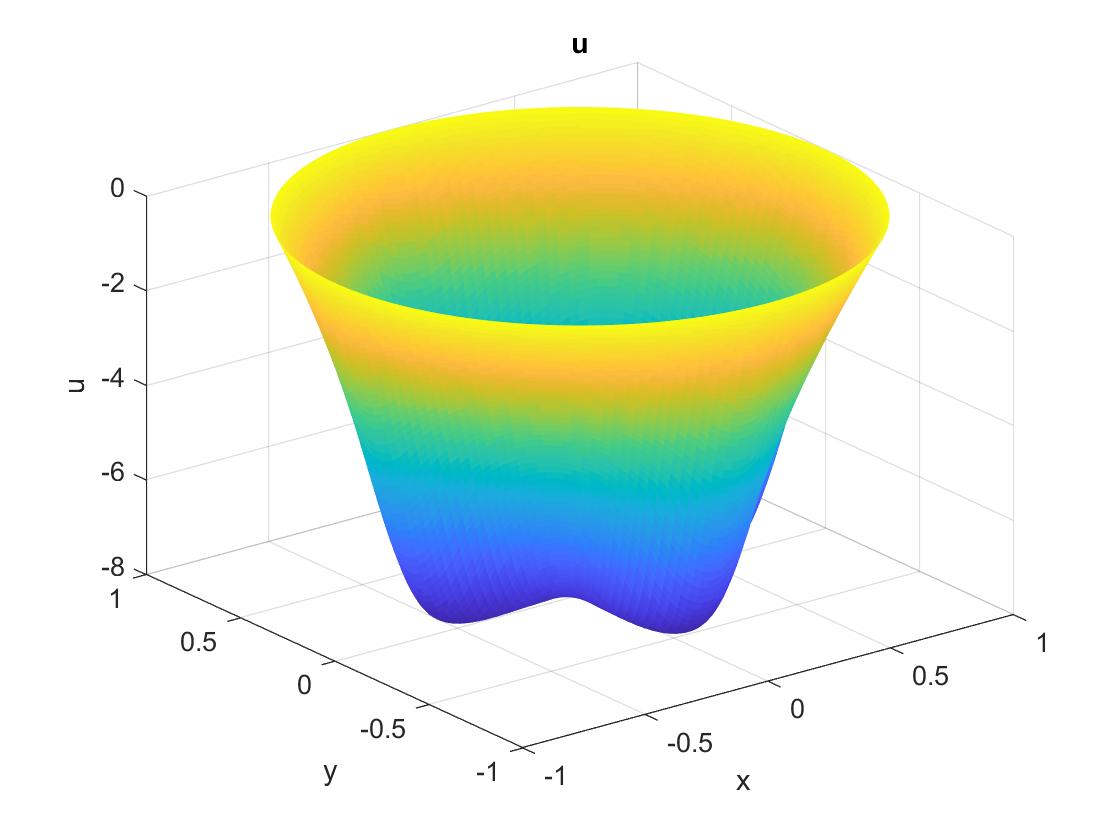}
	\includegraphics[width=0.45\textwidth]{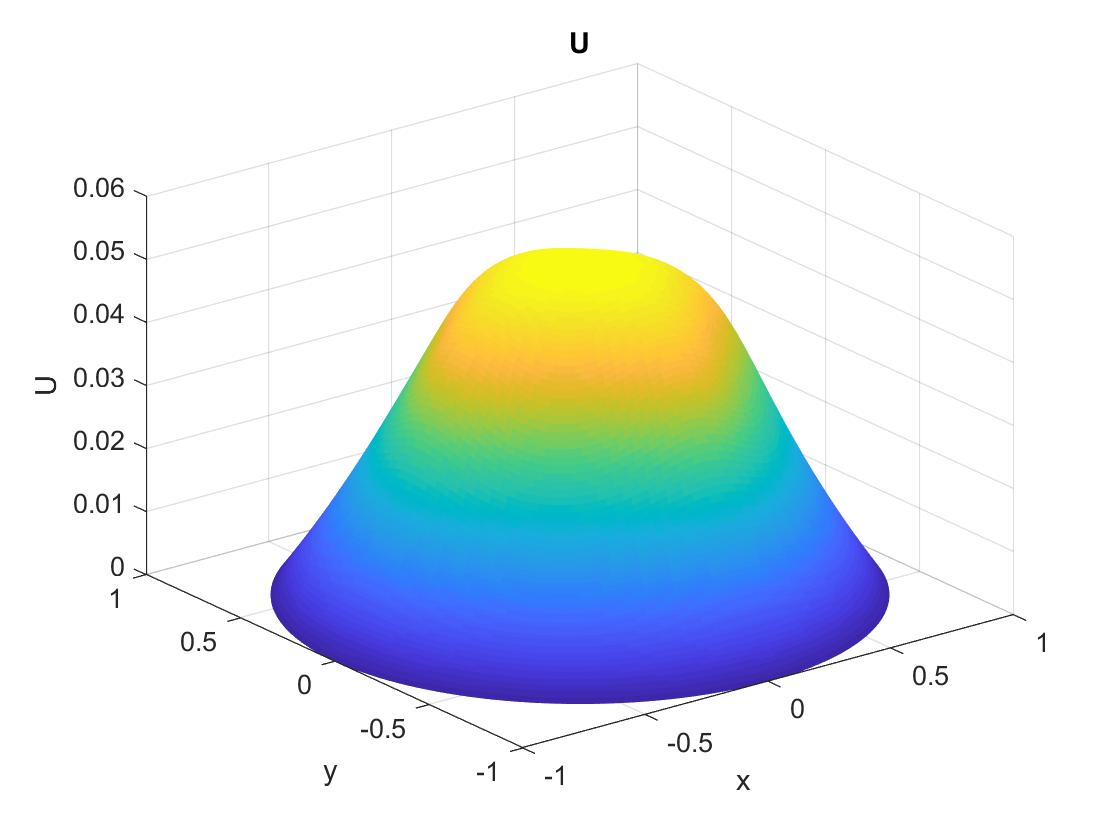}
	\caption{Left: External potential. Right: Self-consistent potential. The figure refers to Example 3.	}
	\label{example_3_2}
\end{figure}

We can see in the Figures \ref{example_1_1}, \ref{example_2_1} and \ref{example_3_1}, that in all cases, the action of the external control field 
concentrates the plasma around the region where the valley function reaches its lower values, and its distribution reflects the geometrical properties of this function. 
On the other hand, without a control field the charge is mainly concentrated near the boundary of the domain as shown in Fig. \ref{fig_rho_0}.

\section*{Conclusion}

This work was devoted to the formulation, theoretical investigation, 
and numerical solution of an optimal design problem for 
steady equilibrium solutions of the Vlasov-Poisson system where the 
design function is an external electric field. 

The resulting optimization framework required to minimize an ensemble 
objective functional with Tikhonov regularization of the control field 
under the differential constraint of the nonlinear 
elliptic Poisson-Boltzmann equation that models equilibrium solutions of the Vlasov-Poisson system. 
Existence of optimal control fields and their characterization as solutions 
to first-order optimality conditions were discussed, and numerical approximation 
and optimization schemes were developed in order 
to compute these solutions. Results of numerical experiments were 
presented that successfully validated the proposed framework. 

\section*{Acknowledgements}
{The authors would like to thank very much the three anonymous reviewers 
who have, with their comments and suggestions, helped improve the exposition 
and content of this manuscript.}

A. Borz\`i has been partially supported by the VIS Program of the University of Calabria. This paper was partialy written during some visits of  A. Borz\`i to the Dipartimento di Matematica e Informatica of the Universit\`a della Calabria.
G.~Infante is a member of the Gruppo Nazionale per l'Analisi Matematica, la Probabilit\`a e le loro Applicazioni (GNAMPA) of the Istituto Nazionale di Alta Matematica (INdAM) and of the UMI Group TAA  ``Approximation Theory and Applications''. G.~Mascali is a member of the Gruppo Nazionale per la Fisica Matematica (GNFM) of INdAM. G.~Infante and G.~Mascali are supported by the project POS-CAL.HUB.RIA.  G. Infante was partly funded by the Research project of MUR - Prin 2022 “Nonlinear differential problems with applications to real phenomena” (Grant Number: 2022ZXZTN2). G. Mascali acknowledges the support from MUR, Project PRIN
“Transport phonema in low dimensional structures: models, simulations and theoretical aspects”
CUP E53D23005900006.

\bibliographystyle{unsrt}
\bibliography{references}

\end{document}